\documentclass[a4paper,11pt]{amsart}

\usepackage{amssymb,amsmath,amsthm}
\usepackage[all]{xy}
\usepackage{hhline}

\allowdisplaybreaks

{\theoremstyle{definition}\newtheorem{df}{Definition}[section]}
{\theoremstyle{definition}\newtheorem{rem}[df]{Remark}}
{\theoremstyle{definition}}
{\theoremstyle{definition}\newtheorem{const}[df]{Construction}}
\newtheorem{prop}[df]{Proposition}
\newtheorem{lem}[df]{Lemma}
\newtheorem{thm}[df]{Theorem}
\newtheorem{prob}[df]{Problem}
\newtheorem{cor}[df]{Corollary}

\newcommand{\Z}{\mathbb{Z}}

\newcommand{\R}{\mathbb{R}}
\newcommand{\C}{\mathbb{C}}

\newcommand{\minor}[1]{\vert#1\vert}
\newcommand{\M}{\mathcal{M}}
\newcommand{\X}{\mathcal{X}}
\newcommand{\I}{\mathcal{I}}
\newcommand{\MM}{{}_{\R}{\mathcal{M}}}
\newcommand{\XX}{{}_{\R}{\mathcal{X}}}
\newcommand{\LL}{{}_{\R}{\Lambda}}
\newcommand{\ZZ}{(\Z/2)}
\newcommand{\chmt}[8]{\left(
\begin{array}{ccc}
1 & #1 & #2  \\
#3 &#4 &1  \\
1 &#5 &#6  \\
#7 &#8 &1 
\end{array}
\right)}
\newcommand{\A}{\mathcal{A}}

\title{ON THE CLASSIFICATION OF QUASITORIC MANIFOLDS OVER THE DUAL CYCLIC POLYTOPES}
\author{SHO HASUI}
\date{}
\address{Department of Mathematics, Faculty of Science, Kyoto University, Sakyo-ku, Kyoto 606-8502, Japan}
\email{s.hasui@math.kyoto-u.ac.jp}

\begin{document}

\maketitle

\begin{abstract}
For a simple $n$-polytope $P$,
a quasitoric manifold over $P$ is a $2n$-dimensional smooth manifold with a locally standard action of the $n$-dimensional torus
for which the orbit space is identified with $P$.
This paper shows the topological classification of quasitoric manifolds over the dual cyclic polytope $C^n(m)^*$,
when $n>3$ or $m-n=3$.
Besides, we classify small covers, the ``real version'' of quasitoric manifolds, over all dual cyclic polytopes.
\end{abstract}

\section{INTRODUCTION}
Quasitoric manifolds were introduced in \cite{DJ91} as a topological generalization of non-singular projective toric varieties.
More precisely, a quasitoric manifold is a $2n$-dimensional manifold $M$ with a locally standard $T^n$-action,
for which the orbit space $M/T^n$ is identified with a simple polytope as a manifold with corners.
Here ``locally standard'' means that the $T^n$-action is locally identified with the coordinate-wise multiplication of $T^n$ on $\C^n$,
where $T^n$ is regarded as $T^n=\{ (z_1,\ldots ,z_n)\in \C^n\,\vert\,\vert z_i\vert=1,\,i=1,\ldots,n \}$.
The simplest example is $\C P^n$ with the standard $T^n$-action, whose supporting simple polytope is the $n$-dimensional simplex $\Delta^n$.

There are other examples of quasitoric manifolds in \cite{OR70}.
It was shown that many of compact oriented one-connected four-manifolds with smooth effective $T^2$-actions are quasitoric manifolds.
In fact, the paper also provides the classification of four-dimensional quasitoric manifolds up to diffeomorphism.

We will consider the classifications of quasitoric manifolds over the dual of a cyclic polytope, up to
(i) weakly equivariant homeomorphism, (ii) homeomorphism, and (iii) cohomology equivalence.
Here a ``cohomology equivalence'' between two topological spaces means a graded ring isomorphism between their integral cohomology rings,
which is not necessarily induced from a continuous map.
For a group $G$, a map $f$ between two $G$-spaces is called \textit{weakly equivariant}
if the condition $f(g\cdot x)=\psi(g)\cdot f(x)$ holds for some automorphism $\psi$ of $G$,
where $x$ and $g$ are arbitrary points in the domain and $G$ respectively.

Davis and Januszkiewicz gave the classification of all quasitoric manifolds up to weakly equivariant homeomorphism in \cite{DJ91}.
We will review their construction in the next section and restate the classification (Proposition \ref{prop:cls of qtmfds up to weh}).

There are some results for the classification of quasitoric manifolds up to homeomorphism.
The classification of quasitoric manifolds over $\Delta^n$ is trivial by construction
(see \cite[Example 1.18]{DJ91}).
Using the results of \cite{OR70}, we obtain the classification over convex polygons
(see \cite[Example 1.20]{DJ91}).
Moreover,
Choi, Park and Suh obtained the topological classification of quasitoric manifolds over the product of two simplices $\Delta^n\times\Delta^m$ in \cite{CPS12}.

For a simple polytope $P$, let us denote the set of all weakly equivariant homeomorphism classes of quasitoric manifolds over $P$ by $\M_P$.
Similarly, let $\M_P^{\rm homeo}$ denote the set of all homeomorphism classes, and $\M_P^{\rm coh}$ denote the set of all cohomology equivalence classes.
Then we have a sequence of surjections $\M_P\rightarrow \M_P^{\rm homeo}\rightarrow \M_P^{\rm coh}$.
Let us summarize the classification results with these notations.

\begin{thm}[\cite{DJ91} and \cite{CPS12}]\label{thm:preceding studies}
For quasitoric manifolds over a simplex, a convex polygon, and the product of two simplices, we have the following classification.
Remark that $\Delta^1\times\Delta^1$ is equal to the convex tetragon $P_4$.
\begin{enumerate}
 \item[$(1)$] For the $n$-dimensional simplex $\Delta^n$, $\#\M_{\Delta^n}=\#\M_{\Delta^n}^{\rm homeo}=\#\M_{\Delta^n}^{\rm coh}=1$.
 \item[$(2)$] For the convex $m$-gon $P_m$, where $m\geq 4$, $\M_{P_m}$ has countably infinite elements.
 If $m$ is odd, $\#\M_{P_m}^{\rm homeo}=\#\M_{P_m}^{\rm coh}=(m-1)/2$.
 If $m$ is even, $\#\M_{P_m}^{\rm homeo}=\#\M_{P_m}^{\rm coh}=m/2+1$.
 \item[$(3)$] For $\Delta^n\times\Delta^m$, where $n\geq 1$ and $m\geq 2$, $\M_{\Delta^n\times\Delta^m}^{\rm coh}$ has countably infinite elements.
 $\M_{\Delta^n\times\Delta^m}\rightarrow \M_{\Delta^n\times\Delta^m}^{\rm homeo}$ is not bijective but
 $\M_{\Delta^n\times\Delta^m}^{\rm homeo}\rightarrow \M_{\Delta^n\times\Delta^m}^{\rm coh}$ is.
\end{enumerate}
\end{thm}

Corollary \ref{cor:clsforspx} and Corollary \ref{cor:classforqm} give representatives of the elements of $\M_{\Delta^n}$ and $\M_{P_m}^{\rm homeo}$ respectively.

We will show the following new results for the classification of quasitoric manifolds over the dual cyclic polytope $C^n(m)^*$.
Note that, in the case where $n\leq 2$ or $m-n\leq 2$,
the topological classification of quasitoric manifolds is already known since $C^2(m)^*$ is the convex $m$-gon and $C^n(n+2)^*$ is $\Delta^l\times\Delta^{n-l}$,
where $l$ denotes the largest integer not more than $n/2$.

\begin{thm}\label{thm:no existence of qtmfds}
If $n\geq 4$ and $m-n\geq 4$, or $n\geq 6$ and $m-n\geq 3$, there exists no quasitoric manifold over $C^n(m)^*$.
\end{thm}

\begin{thm}\label{thm:summary}
For quasitoric manifolds over the dual cyclic polytope $C^n(n+3)^*$, we have the following classification.
\begin{enumerate}
 \item[$(1)$] $\M_{C^3(6)^*}\rightarrow \M_{C^3(6)^*}^{\rm homeo}$ is not bijective but
 $\M_{C^3(6)^*}^{\rm homeo}\rightarrow \M_{C^3(6)^*}^{\rm coh}$ is, and they have countably infinite elements.
 \item[$(2)$] $\#\M_{C^4(7)^*}=\#\M_{C^4(7)^*}^{\rm homeo}=\#\M_{C^4(7)^*}^{\rm coh}=4$.
 \item[$(3)$] $\#\M_{C^5(8)^*}=\#\M_{C^5(8)^*}^{\rm homeo}=\#\M_{C^5(8)^*}^{\rm coh}=46$.
\end{enumerate}
\end{thm}

Note that the classification of quasitoric manifolds over $C^3(m)^*$ for $m\geq 7$ is still open since it is quite complicated,
and it remains to be a subject for further study.

Theorem \ref{thm:summary} yields a partial affirmative answer for the \textit{cohomological rigidity problem} for quasitoric manifolds
over the dual cyclic polytopes.
A class of topological spaces is called \textit{cohomologically rigid} (for homeomorphism)
if two spaces in the class are homeomorphic if and only if their cohomology rings (with coefficients in an appropriate ring) are isomorphic as graded rings.
Note that the Betti numbers of a quasitoric manifold determines the $f$-vector of the supporting simple polytope
(see Definition \ref{df:f- and h-vector} and Theorem \ref{thm:Betti numbers}).
As a corollary of Theorem \ref{thm:preceding studies},
we obtain the cohomological rigidity of quasitoric manifolds over the simplices, the convex polygons, and the products of two simplices.
There are several other results on the cohomological rigidity for quasitoric manifolds.
For more information, we refer the reader to \cite{CMS11}.

Replacing $T$ by $\Z/2$ and $\C$ by $\R$ in the definition of a quasitoric manifold, we obtain the notion of a small cover.
Any quasitoric manifold has an involution called ``conjugation'', for which the fixed point set is homeomorphic to a small cover.
The \textit{lifting problem} asks whether a small cover can be realized as the fixed point set for the involution on a quasitoric manifold or not.
This problem is affirmative for any small cover of dimension $\leq 3$ (Proposition \ref{prop:lift}).
%
%

For a simple polytope $P$, let us denote the set of all weakly equivariant homeomorphism classes of small covers over $P$ by $\MM_P$.
Similarly, let $\MM_P^{\rm homeo}$ denote the set of all homeomorphism classes, and $\MM_P^{\rm coh}$ denote the set of all $\bmod\,2$ cohomology equivalence classes.
We will show the following theorems.

\begin{thm}\label{thm:no existence of smlcvrs}
If $n\geq 4$ and $m-n\geq 4$, or $n\geq 6$ and $m-n\geq 3$, there exists no small cover over $C^n(m)^*$.
\end{thm}

\begin{thm}\label{thm:summary2}
For small covers over the dual cyclic polytope $C^n(m)^*$, we have the following classification.
\begin{enumerate}
 \item[$(1)$] $\MM_{C^3(m)^*}\rightarrow \MM_{C^3(m)^*}^{\rm homeo}$ is not bijective but
 $\MM_{C^3(m)^*}^{\rm homeo}\rightarrow \MM_{C^3(m)^*}^{\rm coh}$ is.
 \item[$(2)$] $\#\MM_{C^4(7)^*}=\#\MM_{C^4(7)^*}^{\rm homeo}=\#\MM_{C^4(7)^*}^{\rm coh}=1$.
 \item[$(3)$] $\#\MM_{C^5(8)^*}=\#\MM_{C^5(8)^*}^{\rm homeo}=\#\MM_{C^5(8)^*}^{\rm coh}=1$.
\end{enumerate}
\end{thm}

Note that Theorem \ref{thm:no existence of qtmfds} is a corollary of Theorem \ref{thm:no existence of smlcvrs}.
Indeed, if there exists a quasitoric manifold over a simple polytope $P$,
the fixed point set for the conjugation involution provides a small cover over the same polytope $P$.

With Theorem \ref{thm:summary} and Theorem \ref{thm:summary2}, we obtain the following corollary.
\begin{cor}
The lifting problem is affirmative for all small covers over the dual cyclic polytopes.
\end{cor}

Now we briefly describe how this paper is organized.
In Section 2, we recall some basic constructions and properties of quasitoric manifolds and small covers.
In Section 3, we recall some properties of a cyclic polytope, prepare some notations, and introduce the connected sum operation.
Using the connected sum operation, in Section 4,
we give a ``purely quasitoric'' proof for the topological classification of quasitoric manifolds and small covers over $C^2(m)^*$.
In Section 5, we classify quasitoric manifolds over $C^3(6)^*$ and small covers over $C^3(m)^*$, using the connected sum operation again.
In Section 6, we prove Theorem \ref{thm:no existence of smlcvrs} and list all small covers over $C^4(7)^*$ and $C^5(8)^*$.
In Section 7, we list all elements of $\M_{C^4(7)^*}$ and show that their integer cohomology rings are not isomorphic to each other.
In Section 8, we show the classification for $C^5(8)^*$ similarly.

%
%
%

\section{BASIC CONSTRUCTIONS AND PROPERTIES}
The purposes of this section are to define a quasitoric manifold and describe some properties of it.
Let us start with recalling some basic concepts and constructions for convex polytopes.

\subsection{Polytopes}
An $n$-dimensional convex polytope is called \textit{simple}, if exactly $n$ facets meet at each vertex.
We can naturally regard a simple polytope as a manifold with corners, in the following sense:
an \textit{$n$-dimensional manifold with corners} is a Hausdorff space having an atlas of open subsets
and homeomorphisms from them to open subsets of $\R_+^n=\{ (x_1,\ldots,x_n)\in\R^n\,\vert\, x_i\geq 0,\,i=1,\ldots,n \}$,
such that the transition maps preserve the canonical stratification of $\R_+^n$.

Let $P$ be an $n$-dimensional simple polytope with $m$ facets which are labeled by $F_1,\ldots,F_m$.
We define $K_P$, an abstract simplicial complex on $\{1,\ldots,m\}$, as follows:
$$\{i_1,\ldots,i_k\}\in K_P \Longleftrightarrow   F_{i_1}\cap\ldots\cap F_{i_k}\neq \emptyset.$$
$K_P$ is identified with the face poset of $P^*$.
Remark that $\{i_1,\ldots,i_n\}$ is a maximal face of $K_P$ if and only if $F_{i_1},\ldots,F_{i_n}$ meet at a vertex.

The nerve of $K_P$ provides a triangulation of $P$. We can see this as follows:
Fix a point $c_G\in\mathrm{rel.int\,}G$ for each face $G$,
and identify a chain $G_1\supsetneq \ldots \supsetneq G_k$ with the $(k-1)$-simplex spanned by $c_{G_1},\ldots,c_{G_k}$.
Therefore, if there is a simplicial isomorphism between $K_P$ and $K_{P'}$, it induces a homeomorphism between $P$ and $P'$.
We obtain the following proposition.

\begin{prop}\label{prop:polytope}
Let $P$ and $P'$ be simple polytopes.
Then a simplicial isomorphism between $K_P$ and $K_{P'}$ can be realized as a homeomorphism between $P$ and $P'$ as manifolds with corners.
Conversely, a homeomorphism between $P$ and $P'$ as manifolds with corners induces a simplicial isomorphism between $K_P$ and $K_{P'}$.
\end{prop}

\subsection{Definition of a quasitoric manifold}
To define quasitoric manifolds and consider their classification, we recall some basic concepts for $T^n$-manifolds.

The standard action of the $n$-dimensional torus $T^n=(S^1)^n$ on $\C^n$ is given as
\[T^n\times \C^n \rightarrow \C^n \colon (t_1,\ldots ,t_n)\times (z_1,\ldots ,z_n)\mapsto  (t_1z_1,\ldots ,t_nz_n),\]
where we canonically regard $T^n=\{ (z_1,\ldots ,z_n)\in \C^n\,\vert\,\vert z_i\vert=1,\,i=1,\ldots,n \}$.
Note that the orbit space for this action is naturally identified with the positive cone $\R_+^n$.

A \textit{$T^n$-manifold} is a differentiable manifold with a smooth action of $T^n$.

\begin{df}
Let $G$ be a group, $X,Y$ be $G$-spaces, and $f$ be a map from $X$ to $Y$.
$f$ is said to be \textit{weakly equivariant} if the condition $f(g\cdot x)=\psi(g)\cdot f(x)$ holds for any $x\in M$ and $g\in G$,
where $\psi$ is some automorphism of $G$.
\end{df}

Let $M$ be a $2n$-dimensional $T^n$-manifold. A \textit{standard chart} on $M$ is a couple $(U,f)$,
where $U$ is a $T^n$-stable open subset of $M$ and $f$ is a weakly equivariant diffeomorphism from $U$ onto some $T^n$-stable open subset of $\C^n$.
An atlas which consists of standard charts is called a \textit{standard atlas}.
If a $2n$-dimensional $T^n$-manifold has a standard atlas, then we say that the $T^n$-action is \textit{locally standard}.
The orbit space for a locally standard action is naturally regarded as a manifold with corners.
Then we define a quasitoric manifold as below.

\begin{df}[Quasitoric manifold]
A $2n$-dimensional $T^n$-manifold is said to be a \textit{quasitoric manifold over a simple polytope} $P$,
if the action is locally standard and the orbit space is homeomorphic to $P$ as a manifold with corners.
\end{df}

According to Proposition \ref{prop:polytope}, we do not have to distinguish combinatorially equivalent simple polytopes in the above definition.
The following proposition is obvious.

\begin{prop}\label{prop:self-homeo}
Let $M$ and $M'$ be quasitoric manifolds over simple polytopes $P$ and $P'$ respectively.
A weakly equivariant homeomorphism from $M$ to $M'$ descends to a homeomorphism from $P$ to $P'$ as manifolds with corners.
\end{prop}

\subsection{Classification up to weakly equivariant homeomorphism}
The notion of a characteristic matrix plays the central role in the classification of quasitoric manifolds.
Let $P$ be an $n$-dimensional simple polytope with $m$ facets $F_1,\ldots,F_m$.

\begin{df}
A \textit{characteristic matrix on} $P$ is an integer $(n\times m)$-matrix $\lambda=({\boldsymbol\lambda}_1,\ldots ,{\boldsymbol\lambda}_m)$
satisfying the \textit{non-singular condition for} $P$:
if $n$ facets $F_{i_1},\ldots ,F_{i_n}$ of $P$ meet at a vertex, then $\det \lambda_{(i_1,\ldots ,i_n)}=\pm 1$,
where $\lambda_{(i_1,\ldots ,i_n)}:=({\boldsymbol\lambda}_{i_1},\ldots ,{\boldsymbol\lambda}_{i_n})$.
\end{df}

From each quasitoric manifold $M$ over $P$, we obtain the corresponding characteristic matrix $\lambda(M)$ as follows.
Let $\pi$ denote the projection from $M$ to $P\cong M/T^n$.
By definition, we see that every point in $\pi^{-1}(\mathrm{rel.int\,}(F_i))$ has the same isotropy subgroup which is a one-dimensional subtorus of $T^n$.
We denote it by $T_M(F_i)$.
Then, taking primitive vectors ${\boldsymbol\lambda}_i={}^t(\lambda_{1,i},\ldots,\lambda_{n,i})\in\Z^n\,(i=1,\ldots,m)$ such that
$$T_M(F_i)=\{(t^{\lambda_{1,i}},\ldots,t^{\lambda_{n,i}})\in T^n\,\vert\,t\in\C,\vert t \vert =1 \},$$
we obtain an integer $(n\times m)$-matrix $\lambda(M):=(\lambda_{i,j})$.
Clearly, each ${\boldsymbol\lambda}_i$ is determined up to sign.
Since the $T^n$-action on $M$ is locally standard, $\lambda(M)$ satisfies the non-singular condition for $P$.

Conversely, we can construct a quasitoric manifold $M(\lambda)$ from a simple polytope $P$ and a characteristic matrix $\lambda$ on $P$.

\begin{const}\label{const:construction}
For each point $q\in P$, we denote the minimal face containing $q$ by $G(q)$.
We define the \textit{characteristic map $\ell_{\lambda}$ corresponding to} $\lambda$,
a map from the set of the faces of $P$ to the set of subtori of $T^n$, as
\[ \ell_\lambda(F_{i_1}\cap \ldots \cap F_{i_k}):= \mathrm{im}\,(\lambda_{(i_1,\ldots ,i_k)}\colon T^k \rightarrow T^n), \]
where we regard $T$ as $\R/\Z$ and $\lambda_{(i_1,\ldots ,i_k)}$ as the map induced from the linear map determined by $\lambda_{(i_1,\ldots ,i_k)}$.
Then we obtain a quasitoric manifold $M(\lambda)$ over $P$ by setting
\[M(\lambda):=(T^n\times P)/\sim_{\lambda},\]
where $(t_1,p)\sim_{\lambda}(t_2,q)$ if and only if $p=q$ and $t_1t_2^{-1}\in \ell_{\lambda}(G(q))$.
There is a standard differentiable structure given by the atlas $\{\pi^{-1}(U_v)\}$, where the index $v$ runs over all vertices of $P$ and
$U_v$ is the open subset obtained by deleting all faces not containing $v$ from $P$.
Indeed, each $U_v$ is homeomorphic to an open subset of $\R_+^n$ by an affine map,
and hence each $\pi^{-1}(U_v)$ is weakly equivariantly homeomorphic to an open subset of $\C^n$.
We easily see that the transition maps are diffeomorphic.
It is obvious that the $T^n$-action on $T^n\times P$ by multiplication on the first component descends to a locally standard action on $M(\lambda)$,
for which the orbit space is identified with $P$.
By definition, $\lambda$ is the corresponding characteristic matrix to $M(\lambda)$.
\end{const}

\begin{df}
Let us denote the set of all characteristic matrices on a simple polytope $P$ by $\Lambda_P$.
As in Section 1, $\M_P$ denotes the set of all weakly equivariant homeomorphism classes of quasitoric manifolds over $P$.
Then Construction \ref{const:construction} gives the map $\Lambda_P\rightarrow\M_P\colon \lambda\mapsto M(\lambda)$.
We denote this map by $\phi$.
\end{df}

The following lemma is essential to the classification of quasitoric manifolds up to weakly equivariant homeomorphism
(see \cite[Lemma 1.4]{DJ91} and \cite[p.\,344]{Dav78}).

\begin{lem}[Davis and Januszkiewicz]\label{lem:section}
For any quasitoric manifold $M$ over a simple polytope $P$, there exists a continuous map $s\colon P\rightarrow M$ such that $\pi\circ s=id_P$.
\end{lem}

Then, for any quasitoric manifold over $P$,
the map $T^n\times P\rightarrow M\colon (t,q)\mapsto t\cdot s(q)$ induces an equivariant homeomorphism from $M(\lambda(M))$ to $M$.
We obtain the following proposition.

\begin{prop}[Davis and Januszkiewicz]
For any quasitoric manifold over a simple polytope $P$, $M(\lambda(M))$ is equivariantly homeomorphic to $M$.
In particular, the map $\phi\colon\Lambda_P\rightarrow\M_P$ is surjective.
\end{prop}

Next we consider when two characteristic matrices give rise to weakly equivariantly homeomorphic quasitoric manifolds.
Let us denote the group of simplicial automorphisms of $K_P$ by $\mathrm{Aut\,}(K_P)$.

\begin{df}
$GL(n,\Z)$ acts on $\Lambda_P$ by left multiplication.
Similarly, $(\Z/2)^m$ acts on $\Lambda_P$ by multiplication with $-1$ on each column,
and $\mathrm{Aut\,}(K_P)$ acts by the column permutation.
We denote $GL(n,\Z)\backslash \Lambda_P/(\Z/2)^m$, the quotient of the biaction, by $\X_P$.
Then the right $\mathrm{Aut\,}(K_P)$-action on $\Lambda_P$ descends to an action on $\X_P$.
\end{df}

Let us show that $\phi\colon\Lambda_P\rightarrow\M_P$ descends to $\overline{\phi}\colon\X_P/\mathrm{Aut\,}(K_P)\rightarrow\M_P$.
Let $\lambda=({\boldsymbol\lambda}_1,\ldots ,{\boldsymbol\lambda}_m)$, $\lambda'=({\boldsymbol\lambda}'_1,\ldots ,{\boldsymbol\lambda}'_m)$
be characteristic matrices on a simple polytope $P$, and put $M:=M(\lambda)$, $M':=M(\lambda')$.
Assume that there are $\sigma\in\mathrm{Aut\,}(K_P)$ and $A\in GL(n,\Z)$ such that
\[{\boldsymbol\lambda}'_{\sigma(i)} = \pm A{\boldsymbol\lambda}_i \ (i=1,\ldots ,m).\]
Denote the realization of $\sigma$ by $\overline{f}\colon P\rightarrow P$ (recall Proposition \ref{prop:polytope})
and the automorphism induced from $A$ by $\psi\colon T^n\rightarrow T^n$.
Then we obtain a weakly equivariant homeomorphism $f\colon M\rightarrow M'$ as follows:
$\tilde{f}$ defined by
$$T^n\times P\rightarrow T^n\times P\colon (t,q)\mapsto (\psi(t),\overline{f}(q))$$
descends to a homeomorphism $f\colon M\rightarrow M'$, which is weakly equivariant.

Conversely, assume that there is a weakly equivariant homeomorphism $f$ from $M$ to $M'$ for an automorphism $\psi$ of $T^n$.
Let us denote the matrix corresponding to $\psi$ by $A\in GL(n,\Z)$,
and the self-homeomorphism of $P$ induced from $f$ (recall Proposition \ref{prop:self-homeo}) by $\overline{f}$.
Then we have
\[\psi(T_{M}(F_i))=T_{M'}(\overline{f}(F_i))\ (i=1,\ldots ,m),\]
in other words,
\[{\boldsymbol\lambda}'_{\sigma(i)} = \pm A{\boldsymbol\lambda}_i \ (i=1,\ldots ,m),\]
where $\sigma$ is the simplicial automorphism of $K_P$ defined by
$F_{\sigma(i)} = \overline{f}(F_i)$ for $i=1,\ldots ,m$.
Hence we see that $\overline{\phi}\colon \X_P/\mathrm{Aut\,}(K_P)\rightarrow\M_P$ is injective.
Now we obtain the following proposition.

\begin{prop}[Davis and Januszkiewicz]\label{prop:cls of qtmfds up to weh}
For any simple polytope $P$,
$\phi\colon \Lambda_P\rightarrow\M_P\colon \lambda \mapsto M(\lambda)$ descends to a bijection $\overline{\phi}\colon \X_P/\mathrm{Aut\,}(K_P)\rightarrow\M_P$.
\end{prop}

\begin{df}
Let $P$ be a simple polytope.
We say that two characteristic matrices on $P$ are \textit{equivalent} if they are equal in $\X_P/\mathrm{Aut\,}(K_P)$.
\end{df}

\subsection{Moment-angle manifold}
Sometimes it is more useful to regard a quasitoric manifold over $P$
as the orbit space for a free $T^{m-n}$-action on the \textit{moment-angle manifold} $\mathcal{Z}_P$,
than to regard it as a product of Construction \ref{const:construction} (for details, see \cite{BP02}).
The moment-angle manifold corresponding to a simple $n$-polytope $P$ with $m$ facets is, as a topological space, a subset of $(D^2)^m$ as follows:
$$ \mathcal{Z}_P := \{ z=(z_1,\ldots,z_m)\in (D^2)^m \,\vert\, \sigma_z \in K_P \}, $$
where $D^2$ denotes the unit closed disk in $\C$, and $\sigma_z:=\{i\,\vert\, \vert z_i \vert <1 \}$.
For a characteristic matrix $\lambda$ on $P$, let $T_\lambda$ be the kernel of the epimorphism from $T^m$ to $T^n$ induced from $\lambda$.
Then $T_\lambda$ is an $(m-n)$-dimensional subtorus of $T^m$, which acts freely on $\mathcal{Z}_P$ by the restriction of the coordinate-wise $T^m$-action.
The $T^m$-action is smooth with respect to the natural differentiable structure on $\mathcal{Z}_P$.
We can show that the orbit space $\mathcal{Z}_P/T_\lambda$ is equivariantly homeomorphic to $M(\lambda)$,
with respect to the action of $T^m/T_\lambda$ which is identified with $T^n$ through the epimorphism $\lambda\colon T^m\rightarrow T^n$.

For example, there is a characteristic matrix $\lambda=(-I\,\vert\,1)$ on $\Delta^n$,
where $I$ denotes the unit matrix of size $n$ and $1$ denotes the column vector with all components $1$.
$\mathcal{Z}_{\Delta^n}=S^{2n+1}$ and the one-dimensional torus $T_\lambda$ acts diagonally,
so we see that $M(\lambda)$ is equivariantly homeomorphic to $\C P^n$ with the standard $T^n$-action.

\subsection{Small covers}
Let $P$ be an $n$-dimensional simple polytope again.
Replacing $T^n$ by $(\Z/2)^n$ and $\C^n$ by $\R^n$ in the definition of a quasitoric manifold,
we obtain the notion of a \textit{small cover over} $P$ (for details, see \cite{DJ91}).
Moreover, replacing $\Z$ by $\Z/2$, we define a \textit{real characteristic matrix on} $P$.
Then we can construct a small cover ${}_\R M(\lambda)$ from a real characteristic matrix $\lambda$, in the same way as Construction \ref{const:construction}.
We denote the set of all real characteristic matrices on $P$ by $\LL_P$.
Note that $GL(n,\Z/2)$ acts on $\LL_P$ by left multiplication, and $\mathrm{Aut\,}(K_P)$ acts by the column permutation.
As in Section 1, we denote the set of all weakly equivariant homeomorphism classes of small covers over $P$ by $\MM_P$

\begin{df}
Let us denote the map $\LL_P\rightarrow\MM_P\colon\lambda\mapsto {}_\R M(\lambda)$ by ${}_\R\phi$,
and the quotient $GL(n,\Z/2)\backslash \LL_P$ by $\XX_P$.
Then the right $\mathrm{Aut\,}(K_P)$-action on $\LL_P$ descends to an action on $\XX_P$.
We say that two real characteristic matrices on $P$ are \textit{equivalent} if they are equal in $\XX_P/\mathrm{Aut\,}(K_P)$.
\end{df}

As for quasitoric manifolds, we get the following proposition for small covers.

\begin{prop}\label{prop:sc1}
For any simple polytope $P$,
${}_\R\phi\colon\LL_P\rightarrow\MM_P\colon\lambda \mapsto {}_\R M(\lambda)$ descends to a bijection
${}_\R\overline{\phi}\colon\XX_P/\mathrm{Aut\,}(K_P)\rightarrow\MM_P$.
\end{prop}

Replacing $D^2$ by the interval $[-1,1]$ and $S^1$ by $\{1,-1\}$, we define the \textit{real moment-angle manifold} ${}_\R\mathcal{Z}_P$.
For a real characteristic matrix $\lambda$ on $P$, let $S_\lambda$ be the kernel of the epimorphism from $\ZZ^m$ to $\ZZ^n$ induced from $\lambda$.
In the same way as $\mathcal{Z}_P$, we see that the orbit space ${}_\R\mathcal{Z}_P/S_\lambda$ is equivariantly homeomorphic to ${}_\R M(\lambda)$,
with respect to the action of $\ZZ^m/S_\lambda$ which is identified with $\ZZ^n$ through the epimorphism $\lambda\colon \ZZ^m\rightarrow \ZZ^n$.
For example, there is a real characteristic matrix $\lambda=(I\,\vert\,1)$ on $\Delta^n$.
${}_\R\mathcal{Z}_{\Delta^n}=S^{n}$ and $S_\lambda$ acts diagonally,
so we see that ${}_\R M(\lambda)$ is equivariantly homeomorphic to $\R P^n$ with the standard $\ZZ^n$-action.

Then we have the classification of quasitoric manifolds and small covers over a simplex, as below.

\begin{cor}\label{cor:clsforspx}
For any positive integer $n$, $\M_{\Delta^n}$ (resp. $\MM_{\Delta^n}$) has only one element,
and it is represented by $\C P^n$ (resp. $\R P^n$) with the standard $T^n$-action (resp. $\ZZ^n$-action).
\end{cor}

Let $P$ be a simple $n$-polytope.
If $\lambda$ is a characteristic matrix on $P$, then the modulo $2$ reduction $\overline{\lambda}$ is a real characteristic matrix on $P$.
The involution $T^n\times P\rightarrow T^n\times P\colon (t,q)\mapsto (t^{-1},q)$ descends to an involution on $M(\lambda)$,
and we call it the \textit{conjugation involution} on $M(\lambda)$.
Regarding $\ZZ^n \subseteq T^n$ canonically, we easily see that the fixed point set for the conjugation involution is $\ZZ^n$-stable,
and it is equivariantly homeomorphic to ${}_\R M(\overline{\lambda})$.
Then we can rephrase the lifting problem as follows.

\begin{prob}
For a real characteristic matrix $\lambda$ on a simple polytope $P$, is there a characteristic matrix $\lambda'$ on $P$
such that $\lambda$ is the modulo $2$ reduction of $\lambda'$?
\end{prob}

For a real characteristic matrix $\lambda=(\lambda_{i,j})$ on $P$,
let us define $\tilde{\lambda}=(\tilde{\lambda}_{i,j})$ by $\tilde{\lambda}_{i,j}\equiv \lambda_{i,j}\bmod 2$ and $\tilde{\lambda}_{i,j}=0,1$ for each $i,j$.
Since the determinant of any three by three matrix with each component zero or one is between $-2$ and $2$,
we see that, if the dimension of $P$ is not more than three, $\tilde{\lambda}$ is a characteristic matrix on $P$.
Then we have the following proposition.

\begin{prop}\label{prop:lift}
The lifting problem is affirmative for any real characteristic matrix on a simple polytope of dimension $\leq 3$.
\end{prop}

\subsection{Cohomology}
For the cohomology rings of a quasitoric manifold and a small cover, the following theorems are known
(\cite[Theorem 3.1]{DJ91} and \cite[Theorem 4.14]{DJ91}, see also \cite{BP02}).

\begin{df}\label{df:f- and h-vector}
Let $P$ be a simple $n$-polytope, and $f_i$ be the number of $(n-i-1)$-dimensional faces of $P$.
Then the integer vector $(f_0,\ldots,f_{n-1})$ is called the \textit{$f$-vector} of $P$.
We define the \textit{$h$-vector} $(h_0,\ldots,h_n)$ of $P$ from the equation
$$ h_0t^n+\cdots+h_{n-1}t+h_n=(t-1)^n+f_0(t-1)^{n-1}+\cdots+f_{n-1} $$
in the polynomial ring $\Z[t]$.
\end{df}

Remark that the $h$-vector of a simple polytope determines the $f$-vector conversely.

\begin{thm}[Davis and Januszkiewicz]\label{thm:Betti numbers}
Let $P$ be a simple $n$-polytope and $(h_0,\ldots,h_n)$ be the $h$-vector of $P$.
\begin{enumerate}
 \item[$(1)$] Suppose that there exists a small cover $M$ over $P$.
 Let $b_i(M)$ be the $i$-th mod $2$ Betti number of $M$, i.e. $b_i(M):=\dim_{\Z/2} H_i(M;\Z/2)$.
 Then $b_i(M)=h_i$.
 \item[$(2)$] Suppose that there exists a quasitoric manifold $M$ over $P$.
 The homology of $M$ vanishes in odd dimensions and is free in even dimensions.
 Let $b_{2i}(M)$ be the $2i$-th Betti number of $M$.
 Then $b_{2i}(M)=h_i$.
\end{enumerate}
\end{thm}

Let $P$ be a simple polytope with $m$ facets $F_1,\ldots,F_m$, and $M$ be a quasitoric manifold (resp. small cover) over $P$.
We denote the projection from $M$ to $P$ by $\pi$.
Then each $\pi^{-1}(F_i)$ is a closed submanifold of dimension $2n-2$ (resp. $n-1$).

\begin{thm}[Davis and Januszkiewicz]\label{thm:coh,quasi}
Let $P$ be an $n$-dimensional simple polytope with $m$ facets, $\lambda=(\lambda_{i,j})$ be a characteristic matrix on $P$, and put $M:=M(\lambda)$.
Then the integral cohomology ring of $M$ is given by
\[ H^*(M;\Z)=\Z[v_1,\ldots,v_m] /(\mathcal{I}_P+\mathcal{J}_{\lambda}),\]
where $v_i\in H^2(M;\Z)$ is the Poincar\'{e} dual of the closed submanifold $\pi^{-1}(F_i)$,
and $\mathcal{I}_P$, $\mathcal{J}_{\lambda}$ are the ideals below:
\begin{align*}
\mathcal{I}_P &= (v_{i_1}\cdots v_{i_k}\,\vert \, F_{i_1}\cap \ldots \cap F_{i_k}=\emptyset ),\\
\mathcal{J}_{\lambda} &= (\lambda_{i,1}v_1+\cdots +\lambda_{i,m}v_m\, \vert \, i=1,\ldots,n).
\end{align*}
\end{thm}

\begin{thm}[Davis and Januszkiewicz]\label{thm:coh,small}
Let $P$ be an $n$-dimensional simple polytope with $m$ facets, $\lambda=(\lambda_{i,j})$ be a real characteristic matrix on $P$,
and put $M:={}_\R M(\lambda)$.
Then the mod $2$ cohomology ring of $M$ is given by
\[ H^*(M;\Z/2)=(\Z/2)[v_1,\ldots,v_m] /(\mathcal{I}_P+\mathcal{J}_{\lambda}),\]
where $v_i\in H^1(M;\Z/2)$ is the Poincar\'{e} dual of the closed submanifold $\pi^{-1}(F_i)$,
and $\mathcal{I}_P$, $\mathcal{J}_{\lambda}$ are the ideals below:
\begin{align*}
\mathcal{I}_P &= (v_{i_1}\cdots v_{i_k}\,\vert \,F_{i_1}\cap \ldots \cap F_{i_k}=\emptyset ),\\
\mathcal{J}_{\lambda} &= (\lambda_{i,1}v_1+\cdots +\lambda_{i,m}v_m\,\vert \,i=1,\ldots,n).
\end{align*}
\end{thm}

Moreover, the following theorem is known for the total Stiefel-Whitney class $w(M)$ and the total Pontrjagin class $p(M)$
of a quasitoric manifold (resp. a small cover) $M$
(\cite[Corollary 6.8]{DJ91}).

\begin{thm}[Davis and Januszkiewicz]\label{thm:SWP}
With the notations in Theorem $\ref{thm:coh,small}$ and Theorem $\ref{thm:coh,quasi}$, we have the following formulae.
\begin{enumerate}
 \item[$(1)$] For a small cover $M={}_{\R}M(\lambda)$,
 $$w(M)=\prod_{i=1}^m (1+v_i)$$
 and $p(M)=1$.
 \item[$(2)$] For a quasitoric manifold $M=M(\lambda)$,
 $$w(M)=\prod_{i=1}^m (1+v_i)$$
 and
 $$p(M)=\prod_{i=1}^m (1-v_i^2).$$
\end{enumerate}
\end{thm}

For the convenience in subsequent arguments, we denote the integral cohomology ring of $M(\lambda)$ by $H^*(\lambda)$ for a characteristic matrix $\lambda$.
Similarly, for a real characteristic matrix $\lambda'$, we denote the mod $2$ cohomology ring of ${}_\R M(\lambda')$ by $H^*(\lambda')$.

%
%
%

\section{CYCLIC POLYTOPES AND CONNECTED SUMS}
First, we review the definition of a cyclic polytope and its combinatorial structure.
See e.g. \cite{Zie95} for details.

\subsection{Cyclic polytopes}
Recall that a convex polytope is called \textit{simplicial} if all facets are simplices.
By definition, the dual of a simplicial polytope is simple and vice-versa.
For a combinatorial polytope $P$, we denote the dual of $P$ by $P^*$.

Given an increasing sequence $t_1<\ldots<t_m$ of real numbers, let $C^n(t_1,\ldots,t_m)$ be the convex hull of $m$ points
$$ T_1:=(t_1,t_1^2,\ldots,t_1^n),\ldots,T_m:=(t_m,t_m^2,\ldots,t_m^n) $$
in $\R^n$ for $n<m$.
Then $C^n(t_1,\ldots,t_m)$ is an $n$-dimensional simplicial polytope with $m$ vertices $T_1,\ldots,T_m$.
The combinatorial structure of $C^n(t_1,\ldots,t_m)$ is characterized as:

\begin{thm}[Gale's evenness condition, {\cite[Theorem 0.7]{Zie95}}]
Let $T$ denote the vertex set $\{ T_1,\ldots,T_m \}$.
Then an $n$-subset $S\subseteq T$ forms a facet of $C^n(t_1,\ldots,t_m)$ if and only if any two elements in $T \setminus S$ are separated by
an even number of elements from $S$ in the sequence $(T_1,\ldots,T_m)$.
\end{thm}

It follows immediately from this theorem that the combinatorial type of $C^n(t_1,\ldots,t_m)$ depends only on $m$ and $n$.
Then we denote this combinatorial polytope by $C^n(m)$ and call it the \textit{cyclic $n$-polytope with $m$ vertices}.
In this paper, we are particularly interested in the dual simple polytope $C^n(m)^*$.
For convenience, we restate Gale's evenness condition in the dual version.

\begin{thm}[dual evenness condition]\label{thm:GEC}
There is an order of the facets of $C^n(m)^*$, say $F_1,\ldots,F_m$, such that
$F_{i_1},\ldots,F_{i_n}$ meet at a vertex if and only if $\{i_1,\ldots,i_n\}$ is the disjoint union of $I_j$'s for some $j$'s in $\{0,\ldots,m\}$,
where $I_j:=\{j,j+1\}\cap\{1,\ldots,m\}$.
\end{thm}

\subsection{Connected sums}
Next, we introduce the connected sum of (real) characteristic matrices.
Let us first recall the connected sum of two combinatorial simple polytopes.
For the formal definition, see \cite{BP02} and \cite[Section 2.6]{Zie95}.
Let $P$ and $Q$ be $n$-dimensional combinatorial simple polytopes with the facets $F_1,\ldots,F_{m_1}$ and $G_1,\ldots,G_{m_2}$ respectively,
and put $F:=\{F_1,\ldots,F_{m_1}\}$, $G:=\{G_1,\ldots,G_{m_2}\}$.
Take two vertices $v\in P$ and $w\in Q$.
Denote the sets of the facets meeting at $v$ and $w$ by $F_v=\{F_{i_1},\ldots,F_{i_n}\}$ and $G_w=\{G_{j_1},\ldots,G_{j_n}\}$ respectively,
and fix a one-to-one correspondence $\tau\colon\{i_1,\ldots,i_n\}\rightarrow \{j_1,\ldots,j_n\}$.
``Cutting off'' the vertices $v$ and $w$ from $P$ and $Q$ respectively, we obtain simple polytopes $P'$ and $Q'$, each of which has the new simplex facet.
Then, after a projective transformation,
we can ``glue'' $P'$ to $Q'$ along the simplex facets, using $\tau$.
Thus we obtain the \textit{connected sum} $P\#_\tau Q$, a simple polytope the set of whose facets can be identified with $F\cup_\tau Q$.
What is important is that $K_{P\#_\tau Q}$ is identified with $K_P\cup_\tau K_Q \setminus \{\sigma\}$ as an abstract simplicial complex on $F\cup_\tau G$,
where $\sigma$ denotes the simplex coming from $\{i_1,\ldots,i_n\}\in K_P$.

Moreover, suppose that there exist (real) characteristic matrices $\lambda=({\boldsymbol\lambda}_1,\ldots,{\boldsymbol\lambda}_{m_1})$
and $\lambda'=({\boldsymbol\lambda'}_1,\ldots,{\boldsymbol\lambda'}_{m_2})$ on $P$ and $Q$ respectively.
If they satisfy the condition ${\boldsymbol\lambda}_{i_k}={\boldsymbol\lambda'}_{\tau(i_k)}$ for $k=1,\ldots,n$,
we obtain the connected sum $\lambda \#_\tau \lambda'$, a (real) characteristic matrix on $P\#_\tau Q$, as follows:
Number the facets of $P\#_\tau Q$ from $1$ to $m:=m_1+m_2-n$, and define $\lambda \#_\tau \lambda'=({\boldsymbol\lambda''}_1,\ldots,{\boldsymbol\lambda''}_m)$
by ${\boldsymbol\lambda''}_k:={\boldsymbol\lambda}_i$ if the $k$-th facet of $\lambda \#_\tau \lambda'$ corresponds to $F_i$,
and ${\boldsymbol\lambda''}_k:={\boldsymbol\lambda'}_j$ if the $k$-th facet of $\lambda \#_\tau \lambda'$ corresponds to $G_j$.
It is easy to observe that $M(\lambda \#_\tau \lambda')$ (resp. ${}_\R M(\lambda \#_\tau \lambda')$)
is homeomorphic to $M(\lambda) \# M(\lambda')$ (resp. ${}_\R M(\lambda) \# {}_\R M(\lambda')$)
with respect to the proper orientations of $M(\lambda)$ and $M(\lambda')$.

Conversely, suppose that a simple $n$-polytope $P$ is decomposed into the connected sum of simple polytopes $P',P''$.
Denote the facets of $P$ by $F_1,\ldots,F_m$.
For simplicity, we assume that the facets $F_1,\ldots,F_k$ come from $P'$ and $F_{k-n+1},\ldots,F_m$ from $P''$.
Let $\lambda=({\boldsymbol\lambda}_1,\ldots,{\boldsymbol\lambda}_m)$ be a (real) characteristic matrix.
If $\lambda$ satisfies the condition $\det ({\boldsymbol\lambda}_{k-n+1},\ldots,{\boldsymbol\lambda}_k)=\pm 1$,
then $\lambda=\lambda'\#_\tau \lambda''$,
where $\lambda'=({\boldsymbol\lambda}_1,\ldots,{\boldsymbol\lambda}_k)$ and
$\lambda''=({\boldsymbol\lambda}_{k-n+1},\ldots,{\boldsymbol\lambda}_m)$ are (real) characteristic matrices on $P'$ and $P''$ respectively,
and $\tau \colon \{k-n+1,\ldots,k\} \rightarrow \{1,\ldots,n\}$ maps $i$ to $i+n-k$.
Thus we have the following lemma.

\begin{lem}\label{lem:connected sum}
Let $P$ be a simple $n$-polytope with $m$ facets $F_1,\ldots,F_m$, and $\lambda$ be a (real) characteristic matrix on $P$.
Suppose that $P$ is decomposed into the connected sum of two simple polytopes $P'$, $P''$,
and the $n$ facets $F_{i_1},\ldots,F_{i_n}$ come from both $P'$ and $P''$.
If the condition $\det \lambda_{(i_1,\ldots,i_n)}=\pm 1$ holds,
then $\lambda$ is decomposed into the connected sum of (real) characteristic matrices on $P'$ and $P''$.
In particular, $M(\lambda)$ (resp. ${}_\R M(\lambda)$) is decomposed into the connected sum of quasitoric manifolds (resp. small covers) over $P'$ and $P''$.
\end{lem}

\subsection{Conventions and remarks}
For subsequent sections, we make the following conventions.
\begin{enumerate}
 \item[(1)] If $m-n>2$, we always label the facets of $C^n(m)^*$ by $F_1,\ldots ,F_m$ so that the dual evenness condition (Theorem \ref{thm:GEC}) works.
 \item[(2)] For an $(n\times m)$-matrix $\lambda$, we denote the minor $\det \lambda_{(i_1,\ldots ,i_n)}$ by $\minor{i_1,\ldots,i_n}_\lambda$,
 or simply $\minor{i_1,\ldots,i_n}$ if $\lambda$ is obvious.
 \item[(3)] For a simple polytope $P$, we identify $\M_P$ (resp. $\MM_P$) with $\X_P/\mathrm{Aut}\,(K_P)$ (resp. $\XX_P/\mathrm{Aut}\,(K_P)$).
 Then, for an element $\M$ of $\M_P$ (resp. $\MM_P$), we say that a (real) characteristic matrix $\lambda$ is a representative of $\M$
 if $M(\lambda)$ (resp. ${}_\R M(\lambda)$) is a representative of $\M$.
\end{enumerate}

Remark that, since $F_1,\ldots,F_n$ meet at a vertex in $C^n(m)^*$, we can take a representative of each element of $\X_{C^n(m)^*}$ (resp. $\XX_{C^n(m)^*}$)
in the form $(I\,\vert\,*)$, where $I$ denotes the unit matrix of size $n$ and $*$ denotes some $(n\times (m-n))$-matrix.
We will basically take such a representative for an element of $\X_{C^n(m)^*}$ (resp. $\XX_{C^n(m)^*}$).

For a characteristic matrix $\lambda=(I\,\vert\,*)$ on $C^n(m)^*$,
the ideal $\mathcal{J}_{\lambda}$ in Theorem \ref{thm:coh,quasi} is reduced to
$$ \mathcal{J}_{\lambda} = (v_i+\lambda_{i,n+1}v_{n+1}+\cdots +\lambda_{i,m}v_m\, \vert \, i=1,\ldots,n). $$
Recall that a \textit{missing face} of an abstract simplicial complex $K$ on $\{ 1.\ldots,m \}$ is
a minimal subset of $\{ 1.\ldots,m \}$ which does not belong to $K$.
Let us denote the set of all missing faces of $K$ by $\mathrm{Miss}\,(K)$.
Then, for any simple polytope $P$, the ideal $\mathcal{I}_{P}$ in Theorem \ref{thm:coh,quasi} is reduced to
$$ \mathcal{I}_P = (v_{i_1}\cdots v_{i_k}\,\vert \, \{i_1,\ldots,i_k\}\in \mathrm{Miss}\,(K_P)). $$
Thus we can restate the theorem as follows.

\begin{thm}\label{thm:cohomology,particular}
Let $\lambda=(\lambda_{i,j})=(I\,\vert\,*)$ be a characteristic matrix on $C^n(m)^*$.
Put $v'_i:=\lambda_{i,n+1}v_{n+1}+\cdots +\lambda_{i,m}v_m \in \Z[v_{n+1},\ldots,v_m]$ for $i=1,\ldots,n$,
and $v'_i:=v_i$ for $i=n+1,\ldots,m$.
Then the cohomology ring $H^*(\lambda)=H^*(M;\Z)$ is given by
\[ H^*(\lambda)=\Z[v_{n+1},\ldots,v_m] /\mathcal{I}_{\lambda},\]
where each $v_i$ has degree two and $\mathcal{I}_\lambda$ is the ideal below:
$$ \I_\lambda = (v'_{i_1}\cdots v'_{i_k}\,\vert \, \{i_1,\ldots,i_k\}\in \mathrm{Miss}\,(K_P)). $$
\end{thm}

We use this expression of the cohomology ring in subsequent sections.

%
%
%

\section{CLASSIFICATION FOR $C^2(m)^*$}
$C^2(m)^*$ is obviously equal to the convex $m$-gon.
As an application of the connected sum operation, we consider the topological classification of quasitoric manifolds and small covers over a convex polygon.
This classification has already appeared in \cite{DJ91}, and for small covers, the proof below goes along the same lines.
Let us suppose that the facets $F_1,\ldots,F_m$ of the convex $m$-gon are numbered clockwise.
We will make use of the following easy lemmas.

\begin{lem}\label{lem:tri1}
Let $\lambda=({\boldsymbol\lambda}_1,\ldots,{\boldsymbol\lambda}_m)$ be a (real) characteristic matrix on the convex $m$-gon.
Suppose that $i,j\in\{1,\ldots,m\}$ ($i<j$) satisfy $\minor{i,j}_\lambda=\pm 1$ and $F_i\cap F_j=\emptyset$.
Then, putting $d:=j-i+1$, $\lambda':=({\boldsymbol\lambda}_i,\ldots,{\boldsymbol\lambda}_j)$ and
$\lambda''=({\boldsymbol\lambda}_j,\ldots,{\boldsymbol\lambda}_m,{\boldsymbol\lambda}_1,\ldots,{\boldsymbol\lambda}_i)$ are
(real) characteristic matrices on the convex $d$-gon and the convex $(m-d+2)$-gon respectively.
Moreover, $\lambda$ is decomposed into the connected sum of $\lambda'$ and $\lambda''$.
\end{lem}
\begin{proof}
It follows immediately from Lemma \ref{lem:connected sum}.
\end{proof}

\begin{lem}\label{lem:tri2}
For any integer matrix $\left(
\begin{array}{cc}
a &c  \\
b &d 
\end{array}
\right)$
of determinant $\pm 1$, if $\vert a \vert < \vert b \vert$, then we have $\vert c \vert \leq  \vert d \vert$.
\end{lem}

First, we consider the classification of small covers.

\begin{thm}\label{thm:decomp of smlcvr over polygon}
If $m>4$, any small cover over a convex polygon is homeomorphic to the connected sum of copies of $S^1\times S^1$ and $\R P^2$.
\end{thm}
\begin{proof}
Let $\lambda$ be a real characteristic matrix on the convex $m$-gon, where $m>4$.
First we show that $\lambda$ is decomposed into the connected sum of real characteristic matrices on two convex polygons.
By the non-singular condition, we have $\minor{2,3}=\minor{1,m}=1$.
Then we can assume that $\lambda$ is in the form
$$\lambda=\left(
\begin{array}{cc|ccc}
1 &0 &1 &\cdots &b  \\
0 &1 &a &\cdots &1 
\end{array}
\right).$$
If $a=1$, we have $\minor{1,3}=1$. If $a=0$, we have $\minor{3,m}=1$.
Then we see that $\lambda$ is decomposed into a connected sum by Lemma \ref{lem:tri1}.

Denote the convex tetragon by $P_4$.
By a direct calculation, we see that there exist exactly two elements in $\MM_{P_4}$, which are represented by
$$\left(
\begin{array}{cc|cc}
1 &0 &1 &0  \\
0 &1 &0 &1 
\end{array}
\right),
\left(
\begin{array}{cc|cc}
1 &0 &1 &1  \\
0 &1 &0 &1 
\end{array}
\right).$$
The left matrix corresponds to the small cover $S^1\times S^1$, and the right is decomposed into a connected sum.
Then we obtain the theorem by induction on $m$.
\end{proof}

By the classification theorem of closed surfaces,
we obtain the topological classification of small covers over a convex polygon as below.

\begin{cor}
A small cover over a convex polygon is homeomorphic to the connected sum of copies of $S^1\times S^1$ or that of $\R P^2$.
\end{cor}

Next, we consider quasitoric manifolds over a convex polygon.
We will show the topological classification of them without using the results of \cite{OR70}.

\begin{lem}
If $m>4$, any characteristic matrix on the convex $m$-gon is decomposed into the connected sum of those on two convex polygons.
\end{lem}
\begin{proof}
As in the proof of Theorem \ref{thm:decomp of smlcvr over polygon}, we can assume that a characteristic matrix on the convex $m$-gon has the form
$$\lambda=\left(
\begin{array}{cc|ccccc}
1 &0 &1 & c_1 &\cdots &c_{m-4} &b  \\
0 &1 &a & d_1 &\cdots &d_{m-4} &1 
\end{array}
\right).$$
If $a=0,\pm 1$, or $b=0,\pm 1$, we obtain the assertion by Lemma \ref{lem:tri1}, so we assume $\vert a\vert >1$ and $\vert b\vert >1$.
Then, by Lemma \ref{lem:tri2}, we have $\vert c_i \vert = \vert d_i \vert$ for some $i$.
Since each column vector is primitive, we have $c_i,d_i=\pm 1$ and therefore $\minor{1,i}=\pm 1$.
Then the proof is completed by Lemma \ref{lem:tri1}.
\end{proof}

By a direct calculation, we see that $\mathcal{M}_{P_4}$ consists of the classes represented by
$$\lambda_k=\left(
\begin{array}{cc|cc}
1 &0 &1 & k  \\
0 &1 &0 & 1  
\end{array}
\right),\,
\lambda'=\left(
\begin{array}{cc|cc}
1 &0 &1 & 2 \\
0 &1 &1 & 1
\end{array}
\right),$$
where $k$ denotes an integer.
We can see that $M(\lambda_k)$ is weakly equivariantly homeomorphic to the Hirzebruch surface $H_k$, and $M(\lambda')$ is homeomorphic to $\C P^2\# \C P^2$.
Recall that the Hirzebruch surface $H_k$ is homeomorphic to $S^2\times S^2$ if $k$ is even, and homeomorphic to $\C P^2 \# \overline{\C P^2}$ if $k$ is odd.
Then we have the following theorem.

\begin{thm}
A quasitoric manifold over a convex polygon is homeomorphic to the connected sum of copies of $\C P^2$, $\overline{\C P^2}$ and $S^2\times S^2$.
\end{thm}

As is well-known, $\C P^2\# (S^2\times S^2)$ is homeomorphic to $\C P^2\#\C P^2\#\overline{\C P^2}$
(it is also shown by computations of characteristic matrices).
Then we obtain the classification below.

\begin{cor}\label{cor:classforqm}
A quasitoric manifold over a convex polygon is homeomorphic to the connected sum of copies of $S^2\times S^2$,
or the connected sum of copies of $\C P^2$ and $\overline{\C P^2}$.
\end{cor}

\begin{rem}
Observing their intersection forms, we see that the connected sums in the statement of Corollary \ref{cor:classforqm} are not homeomorphic to each other,
except the cases below:
$$\underbrace{\C P^2\#\ldots\#\C P^2}_{i}\#\underbrace{\overline{\C P^2}\#\ldots\#\overline{\C P^2}}_{j}
\cong\underbrace{\C P^2\#\ldots\#\C P^2}_{j}\#\underbrace{\overline{\C P^2}\#\ldots\#\overline{\C P^2}}_{i}.$$
\end{rem}

%
%
%

\section{CLASSIFICATION FOR $C^3(m)^*$}
First, we list all maximal faces of $K_{C^3(m)^*}$ below:
$$ \{1,i,i+1\}\,(i=2,\ldots,m-1),\,\{i,i+1,m\}\,(i=1,\ldots,m-2). $$
Let $F_1,\ldots,F_{m_1}$ be the facets of $C^3(m_1)^*$, and $F'_1,\ldots,F'_{m_2}$ of $C^3(m_2)^*$.
Take vertices $v=F_1\cap F_{m_1-1} \cap F_{m_1}$, $w=F'_1\cap F'_2 \cap F'_{m_2}$, and put $m:=m_1+m_2-3$.
The connected sum $C^3(m_1)^*\#_\rho C^3(m_2)^*$ is equal to $C^3(m)^*$,
where $$\rho\colon\{ 1,m_1-1,m_1 \}\rightarrow \{ 1,2,m_2 \}:1\mapsto 1,m_1-1\mapsto 2,m_1\mapsto m_2.$$
We label the facets of $C^3(m_1)^*\#_\rho C^3(m_2)^*$ by $F''_1,\ldots,F''_m$ as follows:
$$F''_1=F_1=F'_1,F''_2=F_2,\ldots,F''_{m_1-2}=F_{m_1-2},F''_{m_1-1}=F_{m_1-1}=F'_2,$$
$$F''_{m_1}=F'_3,\ldots,F''_{m-1}=F'_{m_2-1},F''_m=F_{m_1}=F'_{m_2}.$$
Then the dual evenness condition works for $C^3(m_1)^* \#_\rho C^3(m_2)^*=C^3(m)^*$ with this labeling.
We have the following lemma.

\begin{lem}
Let $\lambda$ be a (real) characteristic matrix on $C^3(m)^*$.
If the condition $\minor{1,k,m}_\lambda=\pm 1$ holds for some $2<k<m-1$,
then $\lambda$ is decomposed into the connected sum of two characteristic matrices on $C^3(k+1)^*$ and $C^3(m-k+2)^*$ respectively.
\end{lem}
\begin{proof}
It follows immediately from Lemma \ref{lem:connected sum} and the above.
\end{proof}

\begin{df}
A (real) characteristic matrix on $C^3(m)^*$ is said to be \textit{decomposable} if the condition $\minor{1,k,m}_\lambda=\pm 1$ holds for some $2<k<m-1$,
and if not, it is said to be \textit{indecomposable}.
Moreover, we say that a quasitoric manifold over $C^3(m)^*$ is decomposable (resp. indecomposable)
if the corresponding (real) characteristic matrix is decomposable (resp. indecomposable).
\end{df}

\begin{lem}\label{lem:Aut on C3(m)}
If $m>5$, $\mathrm{Aut\,}(K_{C^3(m)^*})$ is generated by $\sigma=\left(
\begin{array}{ccccc}
1 &2 &\cdots&m-1 &m  \\
m &m-1 &\cdots&2 &1 
\end{array}
\right)$ and $\tau=(1\ m)$, and if $m=5$, it is generated by $\left(
\begin{array}{ccccc}
1 &2 &3 &4 &5  \\
3 &4 &5 &2 &1 
\end{array}
\right)$ and $(1\ 5)$.
\end{lem}
\begin{proof}
First, we can directly check that these permutations certainly give automorphisms of $K_{C^3(m)^*}$.
Considering the number of vertices of each facets, we see that there exists no other permutation of $\{1,\ldots,5\}$ which can be an automorphism of $K_{C^3(5)^*}$.
Similarly, for any $\rho \in \mathrm{Aut\,}(K_{C^3(m)^*})$ for $m>5$,
we have $\{\rho(1),\rho(m)\}=\{1,m\}$ and $\{\rho(2),\rho(m-1)\}=\{2,m-1\}$.
Composing $\rho$ with $\sigma$ and $\tau$ if necessary, we can assume that $\rho(1)=1,\rho(2)=2,\rho(m-1)=m-1,\rho(m)=m$.
Then we have $\rho(3)=3$ since $\rho(\{1,2,3\})=\{1,2,\rho(3)\}\in K_{C^3(m)^*}$ and $\rho(3)\neq m$.
In the same way, we have $\rho(4)=4,\ldots,\rho(m-2)=m-2$.
Thus the proof is completed.
\end{proof}

Remark that the action of $\mathrm{Aut\,}(K_{C^3(m)^*})$ preserves (in)decomposability of (real) characteristic matrices.
Then we say that an element $\M$ of $\M_{C^3(m)^*}$ (resp. $\MM_{C^3(m)^*}$) is (in)decomposable
if a representative (real) characteristic matrix of $\M$ is (in)decomposable.

To classify quasitoric manifolds over $C^3(6)^*$ topologically, for the first step,
we list all indecomposable characteristic matrices on $C^3(6)^*$ up to equivalence.
Then we will show that the class of indecomposable quasitoric manifolds over $C^3(6)^*$ is cohomologically rigid,
i.e. their homeomorphism classes are distinguished by their cohomology rings.
In the statement of Proposition \ref{prop:list(3,6)}, we mean the whole characteristic matrix $(I\,\vert\, *)$ by $*$.

\begin{prop}\label{prop:list(3,6)}
All indecomposable elements of $\X_{C^3(6)^*}$ are represented by the characteristic matrices below:
\begin{align*}
\lambda_1 &=\left(
\begin{array}{ccc}
0 &0 &1  \\
1 &1 &2  \\
2 &1 &1 
\end{array}
\right), &
\lambda'_1 &=\left(
\begin{array}{ccc}
0 &0 &1  \\
1 &2 &3  \\
1 &1 &1 
\end{array}
\right), \\
\lambda_2 &=\left(
\begin{array}{ccc}
1 &1 &1  \\
1 &1 &2  \\
2 &1 &1 
\end{array}
\right), &
\lambda'_2 &=\left(
\begin{array}{ccc}
-1 &-1 &1  \\
1 &1 &-2  \\
1 &0 &1 
\end{array}
\right), &
\lambda''_2 &=\left(
\begin{array}{ccc}
0 &1 &1  \\
1 &1 &3  \\
1 &0 &1 
\end{array}
\right), &
\lambda'''_2 &=\left(
\begin{array}{ccc}
0 &1 &1  \\
1 &2 &3  \\
1 &1 &1 
\end{array}
\right), \\
\lambda_d &=\left(
\begin{array}{ccc}
0 &0 &1  \\
1 &1 &d  \\
1 &0 &1 
\end{array}
\right),
\end{align*}
where $d\leq -2$ or $d\geq 3$, and the actions of $\sigma,\tau\in \mathrm{Aut}\,(K_{C^3(6)^*})$ are illustrated as follows.
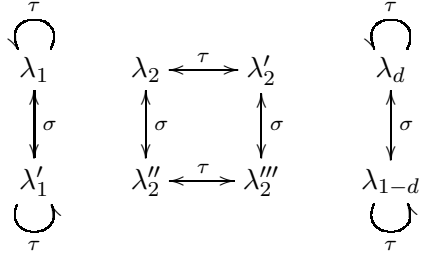
\[
\xymatrix {
\lambda_1 \ar[1,0]^\sigma \ar@(ur,ul)_\tau & \lambda_2 \ar[1,0]^\sigma \ar[0,1]^\tau & \lambda'_2 \ar[0,-1] \ar[1,0]^\sigma
& \lambda_d \ar[1,0]^\sigma \ar@(ur,ul)_\tau \\
\lambda'_1 \ar@(dl,dr)_\tau \ar[-1,0] &\lambda''_2 \ar[-1,0] \ar[0,1]^\tau & \lambda'''_2 \ar[0,-1] \ar[-1,0]
& \lambda_{1-d} \ar@(dl,dr)_\tau \ar[-1,0]
}
\]
\end{prop}
\begin{proof}
Let $\lambda$ be an indecomposable characteristic matrix on $C^3(6)^*$.
Since $\minor{1,3,4}=\pm 1$, $\minor{1,2,6}=\pm 1$ and $\minor{2,3,6}=\pm 1$, we can assume that $\lambda$ has the form
$$\lambda=\left(
\begin{array}{cccccc}
1 &0 &0 &a_1 &b_1 &1  \\
0 &1 &0 &1 &b_2 &c_2  \\
0 &0 &1 &a_3 &b_3 &1 
\end{array}
\right).$$
Moreover, $a_3\neq 0$ from the indecomposability,
and then we can assume $a_3>0$ by multiplication with $-1$ on the second and fourth columns and the middle row if necessary.
Similarly, from the indecomposability, we have $c_2\neq 0,\pm 1$.
If $a_3=1$, we have $c_2\geq 3$ or $c_2\leq -2$ from the indecomposability.
If $a_3\geq 2$, we obtain $b_2,b_3=\pm 1$ by Lemma \ref{lem:tri2}, and hence we have $a_3=2$.
By multiplication with $-1$ on the fifth column if necessary, we can assume $b_2=b_3=1$.

To summarize, we only have to consider two cases, (1) $a_3=1$ and $c_2\geq 3$ or $c_2\leq -2$, and (2) $a_3=c_2=2$ and $b_2=b_3=1$.
By a direct calculation, in the case (1), we obtain the characteristic matrices $\lambda'_1$, $\lambda'_2$, $\lambda''_2$,
$\lambda'''_2$ and $\lambda_d$ for $d\leq -2$ or $d\geq 3$ in the statement.
Similarly, we obtain $\lambda_1$ and $\lambda_2$ in the case (2).
\end{proof}

Then we take $M_1:=M(\lambda'_1)$, $M_2:=M(\lambda'''_2)$ and $M_d:=M(\lambda_d)$ for $d\geq 3$,
and put $A_1:=H^*(\lambda'_1)$, $A_2:=H^*(\lambda'''_2)$ and $A_d:=H^*(\lambda_d)$ for $d\geq 3$.
Then, putting $X:=v_4,Y:=v_5,Z:=v_6$ in Theorem \ref{thm:cohomology,particular}, we have $$A_i=\Z[X,Y,Z]/(\I_i^4+\I_i^6),$$
where $\I_i^4$ and $\I_i^6$ are ideals as follows:
\begin{align*}
\I_1^4=\I_2^4 &= (X(X+2Y+3Z),Y(X+2Y+3Z),Y(X+Y+Z)), \\
\I_1^6 &= (Z^2(X+Y+Z),Z^2X), \\
\I_2^6 &= (Z(Y+Z)(X+Y+Z),ZX(Y+Z)), \\
\I_d^4 &= (X(X+Y+dZ),Y(X+Y+dZ),Y(X+Z)), \\
\I_d^6 &= (Z^2(X+Z),Z^2X).
\end{align*}
Suppose that there exists a graded ring isomorphism from $A_i$ to $A_j$ and denote it by $\phi$.
We may regard $\phi$ as a graded ring automorphism of $\Z[X,Y,Z]$ such that $\phi((\I_i^4+\I_i^6))\subseteq(\I_j^4+\I_j^6)$.
Putting $\phi(X)=a_1X+b_1Y+c_1Z$, $\phi(Y)=a_2X+b_2Y+c_2Z$ and $\phi(Z)=a_3X+b_3Y+c_3Z$,
we write $a=(a_1,a_2,a_3)$, $b=(b_1,b_2,b_3)$ and $c=(c_1,c_2,c_3).$
We mean $\phi$ by the matrix $({}^ta,{}^tb,{}^tc)$.

\begin{lem}
If $i=1,2$ and $j=d\geq 3$, then $c\equiv \pm(1,1,1) \bmod 3$.
\end{lem}
\begin{proof}
Since $Z^2$ does not appear in $I_d^4$, the coefficients of $Z^2$ in $\phi(X(X+2Y+3Z))$, $\phi(Y(X+2Y+3Z))$ and $\phi(Y(X+Y+Z))$ are zero.
Hence we obtain the following equation:
\begin{equation}\label{eq:Z^2}
\left(
\begin{array}{c}
c_1(c_1+2c_2+3c_3) \\
c_2(c_1+2c_2+3c_3) \\
c_2(c_1+c_2+c_3) 
\end{array}
\right)=0.
\end{equation}
Similarly, comparing the coefficients of $X^2$ and the coefficients of $ZX$ in $\phi(X(X+2Y+3Z))$, $\phi(Y(X+2Y+3Z))$ and $\phi(Y(X+Y+Z))$,
we obtain the following equation:
\begin{equation}\label{eq:X^2-ZX}
d \left(
\begin{array}{c}
c_1(a_1+2a_2+3a_3)+a_1(c_1+2c_2+3c_3) \\
c_2(a_1+2a_2+3a_3)+a_2(c_1+2c_2+3c_3) \\
c_2(a_1+a_2+a_3)+a_2(c_1+c_2+c_3)
\end{array}
\right)=\left(
\begin{array}{c}
a_1(a_1+2a_2+3a_3) \\
a_2(a_1+2a_2+3a_3) \\
a_2(a_1+a_2+a_3) 
\end{array}
\right).
\end{equation}
From the equation (\ref{eq:Z^2}), we see that $c\equiv \pm (0,0,1),\pm(1,1,1) \bmod 3$.
If we assume that $c\equiv (0,0,\epsilon) \bmod 3$, where $\epsilon=\pm 1$, the equation (\ref{eq:X^2-ZX}) reduces to
$$d \left(
\begin{array}{c}
0 \\
0 \\
\epsilon a_2
\end{array}
\right) \equiv \left(
\begin{array}{c}
a_1(a_1+2a_2) \\
a_2(a_1+2a_2) \\
a_2(a_1+a_2+a_3) 
\end{array}
\right)\bmod 3.$$
If $a_1\equiv 0 \bmod 3$, one has $2a_2^2\equiv 0 \bmod 3$.
This is a contradiction since $a$ and $c$ are linearly independent modulo $3$, implying $a_1\not\equiv 0 \bmod 3$.
By an analogous argument, we see that $a_2\not\equiv 0 \bmod 3$, implying $a_1\equiv a_2 \bmod 3$ by $a_1+2a_2\equiv 0 \bmod 3$.
\begin{itemize}
 \item In the case that $d\equiv 0 \bmod 3$, $a\equiv \pm(1,1,1) \bmod 3$.
 Then we have $\phi(X-Y)\equiv(b_1-b_2)Y \bmod 3$,
 but this is a contradiction since $(X-Y)^2\equiv 0 \bmod 3$ in $A_i$ and $Y^2\not\equiv 0 \bmod 3$ in $A_d$.
 \item In the case that $d\equiv 1 \bmod 3$,
 $\phi(X(X+2Y))\equiv a_1(b_1+2b_2)XY \bmod (\I_d^4,3)$, but $a_1(b_1+2b_2)XY \notin (\I_d^4,3)$.
 This is a contradiction.
 \item In the case that $d\equiv 2$, there exists no isomorphism $\phi$ since $A_i \bmod 3$ has a homogeneous non-zero element $\alpha$ of degree two such that
 $\alpha^2=0$ but $A_d \bmod 3$ does not.
\end{itemize}
Therefore the proof is completed.
\end{proof}

\begin{prop}
$A_i$ is not isomorphic to $A_d$ for $i=1,2$ and $d\geq 3$.
\end{prop}
\begin{proof}
From the equation (\ref{eq:Z^2}) and the previous lemma, we obtain that $c=\pm(1,-2,1)$.
Considering $-\phi$ instead of $\phi$ if necessary, we can regard $c=(1,-2,1)$.
Then the equation (\ref{eq:X^2-ZX}) reduces to
$$d \left(
\begin{array}{c}
a_1+2a_2+3a_3 \\
-2(a_1+2a_2+3a_3) \\
-2(a_1+a_2+a_3)
\end{array}
\right)=\left(
\begin{array}{c}
a_1(a_1+2a_2+3a_3) \\
a_2(a_1+2a_2+3a_3) \\
a_2(a_1+a_2+a_3) 
\end{array}
\right).$$
Since $a$ and $c$ are linearly independent, we have $a_1+2a_2+3a_3\neq 0$ or $a_1+a_2+a_3\neq 0$.
\begin{itemize}
 \item If $a_1+2a_2+3a_3\neq 0$, we obtain $a_1=d$ and $a_2=-2d$.
 Then, considering the component of $ZX$ in $\phi(X(X+2Y+3Z))\in A_d^4$ with respect to the basis $\{Z^2,YZ,ZX\}$ of $A_d^4$,
 we obtain $(a_3-d)(1-d^2)=0$. Hence $a_3=d$, but this is a contradiction since $a=d(1,-2,1)$.
 \item If $a_1+a_2+a_3\neq 0$, we obtain $a_2=-2d$. If $a_1=d$, we detect a contradiction in the same way as above, so we obtain $a_1+2a_2+3a_3=0$.
 Then we have $$\phi(Y(X+2Y+3Z))\equiv (b_1+2b_2+3b_3)(b_2-2)(1-d)YZ \bmod I_d^4,$$
 so $b_2=2$. However, this is a contradiction since $(a_2,b_2,c_2)=2(-d,1,-1)$.
\end{itemize}
Thus the proof is completed.
\end{proof}

\begin{prop}
For any $d,d'\geq 3$, $A_d$ and $A_{d'}$ are isomorphic if and only if $d=d'$.
\end{prop}
\begin{proof}
Clearly, we only have to consider the case of $d<d'$ and prove that there exists no graded ring isomorphism from $A_d$ to $A_{d'}$.
Let us assume that there exists a graded ring isomorphism $\phi\colon A_d\rightarrow A_{d'}$ as
$$\left(
\begin{array}{c}
\phi(X) \\
\phi(Y) \\
\phi(Z) 
\end{array}
\right)=\left(
\begin{array}{ccc}
a_1 &b_1 &c_1  \\
a_2 &b_2 &c_2  \\
a_3 &b_3 &c_3 
\end{array}
\right)\left(
\begin{array}{c}
X \\
Y \\
Z 
\end{array}
\right),$$ and put $a=(a_1,a_2,a_3),b=(b_1,b_2,b_3),c=(c_1,c_2,c_3)$.
From $\phi(Y(X+Z))=\phi(X(X+Y+dZ))=\phi(Y(X+Y+dZ))=0$ in $A_{d'}$, we obtain the following equations:
\begin{align*}
0 &=c_2(c_1+c_3), \\
0 &=a_2(c_1+c_3)+c_2(a_1+a_3)-d'a_2(a_1+a_3), \\
0 &=c_1(c_1+c_2+dc_3), \\
0 &=a_1(c_1+c_2+dc_3)+c_1(a_1+a_2+da_3)-d'a_1(a_1+a_2+da_3), \\
0 &=c_2(c_1+c_2+dc_3), \\
0 &=a_2(c_1+c_2+dc_3)+c_2(a_1+a_2+da_3)-d'a_2(a_1+a_2+da_3).
\end{align*}

If we assume $c_2\neq 0$, we have $c_1+c_3=c_1+c_2+dc_3=0$, so $c=\pm (-1,1-d,1)$.
If $a_1+a_2+da_3\neq 0$, there occurs a contradiction since we have $d'a_1=\pm 1$.
Hence we obtain $a_1+a_2+da_3=0$. Similarly, $a_1+a_3=0$. However, these equations contradict the linear independence of $a$ and $c$.
Therefore, we obtain $c_2=0$ and $c_1(c_1+dc_3)=0$.

If $c_1=0$, then we have $c=\pm (0,0,1)$.
In the similar way to the above, we can find a contradiction if $a_1\neq 0$ or $a_2\neq 0$.
Hence we obtain $a=\pm (0,0,1)$, but this contradicts the linear independence of $a$ and $c$.
If $c_1+dc_3=0$, we find a contradiction in the same way as above.
\end{proof}

\begin{prop}\label{prop:3,homeo}
$M_1$ and $M_2$ are homeomorphic.
\end{prop}

To prove this proposition, we use the classification theorem for closed, oriented, one-connected $6$-manifolds with torsion-free cohomology in \cite{Jup73}.
While there is the complete classification of such $6$-manifolds, to prove Proposition \ref{prop:3,homeo},
we only need the following special case.

\begin{thm}[\cite{Jup73}]\label{thm:cls for 6-mfds}
Let $M,N$ be closed, one-connected, smooth $6$-manifolds with torsion-free cohomology.
If a graded ring isomorphism $\phi\colon H^*(N;\Z)\rightarrow H^*(M;\Z)$ preserves the second Stiefel-Whitney classes and the first Pontrjagin classes,
then there exists a homeomorphism $f\colon M\rightarrow N$ which induces $\phi$ in cohomology.
\end{thm}

\begin{proof}[Proof of Proposition \ref{prop:3,homeo}]
Let us define a graded ring automorphism $\phi$ of $\Z[X,Y,Z]$ by $$\left(
\begin{array}{c}
\phi(X) \\
\phi(Y) \\
\phi(Z) 
\end{array}
\right)=\left(
\begin{array}{ccc}
1 &1 &0  \\
0 &-1 &0  \\
0 &1 &1 
\end{array}
\right)\left(
\begin{array}{c}
X \\
Y \\
Z 
\end{array}
\right).$$
We easily see that $\phi(\I_1^4+\I_1^6)\subseteq (\I_2^4+\I_2^6)$ and then $\phi$ descends to a graded ring isomorphism from $A_1$ to $A_2$.
Using Theorem \ref{thm:SWP},
we also see that $\phi$ maps the second Stiefel-Whitney class and the first Pontrjagin class of $M_1$ to those of $M_2$.
Hence, by Theorem \ref{thm:cls for 6-mfds},
there exists a homeomorphism between $M_1$ and $M_2$ which induces $\phi$ in cohomology.
\end{proof}

Next, let $\lambda$ be an indecomposable characteristic matrix on $C^3(6)^*$, and $\lambda'$ a decomposable one.

\begin{prop}
$H^*(\lambda)$ and $H^*(\lambda')$ are not isomorphic as graded rings.
\end{prop}
\begin{proof}
First, we consider the case that $H^*(\lambda)=A_i$ for $i=1,2$.
If we assume that $A_i\cong H^*(\lambda')$,
we can take a basis $\{ w_1,w_2,w_3 \}$ of $A_i^2\cong H^2(\lambda')$ such that $w_1w_2=w_1w_3=0$ and $w_1^3$ spans $A_i^6$.
Put $w_i=a_iX+b_iY+c_iZ$ for $i=1,2,3$ and $V=\langle w_2,w_3 \rangle$.
Let us take $\{ YZ,ZX,Z^2\}$ as the basis of $A_i^4$.
Then we have $c_1c_2=c_1c_3=0$ from the coefficients of $Z^2$ in the equation $w_1w_2=w_1w_3=0$.
\begin{itemize}
 \item If $c_1\neq 0$, we have $c_2=c_3=0$ and $c_1=\pm 1$. Hence $X,Y\in V$.
 However, there occurs a contradiction since $0=w_1X\equiv (a_1+b_1)YZ+c_1ZX\not\equiv 0 \bmod 3$.
 \item If $c_1=0$, we obtain $a_1(c_2-3a_1)=a_1(c_3-3a_1)=0$ from the coefficients of $ZX$ in $w_1w_2=w_1w_3=0$.
 Since $(c_1,c_2,c_3)$ is primitive, we obtain $a_1=0$, and then $w_1=\pm Y$.
 However, $Y^3=-2Y^2Z$ in $A_i$, so this contradicts the assumption that $w_1^3$ spans $A_i^6$.
\end{itemize}

Next, we consider the case that $H^*(\lambda)=A_d$ for $d\geq 3$.
Let us assume that there exists a basis $\{ w_1,w_2,w_3 \}$ of $A_d^2$ as above, and also put $w_i=a_iX+b_iY+c_iZ,V=\langle w_2,w_3 \rangle$.
Then we have $c_1c_2=c_1c_3=0$ similarly.
\begin{itemize}
 \item If $c_1\neq 0$, we obtain $c_2=c_3=0$, $c_1=\pm 1$ and $X,Y\in V$.
 Then we have the equation $w_1X=(a_1-b_1)YZ+(c_1-da_1)ZX=0$, but this is a contradiction since $da_1\neq c_1=\pm 1$.
 \item If $c_1= 0$, we obtain $a_1=0$ in the same way as for $A_i\,(i=1,2)$.
 However, $Y^3=(1-d)Y^2Z$ in $A_d$, so this also contradicts the assumption.
\end{itemize}
Thus the proof is completed.
\end{proof}

Finally, we show that two decomposable quasitoric manifolds over $C^3(6)^*$ are homeomorphic if their cohomology rings are isomorphic,
using the cohomological rigidity of quasitoric manifolds over $C^3(5)^*=\Delta^1\times \Delta^2$.
As is remarked in Section 1, this rigidity has been shown in \cite{CPS12}.

\begin{prop}
Let $M=M_1\# M_2$ and $N=N_1\# N_2$, where $M_1,N_1$ are quasitoric manifolds over $C^3(4)^*$ and $M_2,N_2$ over $C^3(5)^*$.
If the cohomology rings of $M$ and $N$ are isomorphic as graded rings, then $M$ and $N$ are homeomorphic.
\end{prop}
\begin{proof}
Remark that $M_1$ and $N_1$ are homeomorphic to $\C P^3$.
We can naturally regard $H^2(M)=H^2(M_1)\oplus H^2(M_2)$.
Then take the basis $u_1,u_2,u_3$ of $H^2(M)$ such that $u_1$ is the basis of $H^2(M_1)$ and $\{u_2,u_3\}$ is that of $H^2(M_2)$.
We take the basis $v_1,v_2,v_3$ of $H^2(N)$ similarly.
Assume that there exists a graded ring isomorphism $\phi\colon H^*(M)\rightarrow H^*(N)$,
and write $\phi(u_i)=a_iv_1+b_iv_2+c_iv_3$ for $i=1,2,3$.
Then the matrix $\left(
\begin{array}{ccc}
a_1 &b_1 &c_1  \\
a_2 &b_2 &c_2  \\
a_3 &b_3 &c_3 
\end{array}
\right)$ is invertible.

\begin{itemize}
 \item If $a_1=0$, 
 from $u_1u_2=u_1u_3=0$ and $v_1v_2=v_1v_3=0$, we have $(b_1v_2+c_1v_3)(b_2v_2+c_2v_3)=(b_1v_2+c_1v_3)(b_3v_2+c_3v_3)=0$.
 Then we easily see that $b_2v_2+c_2v_3$ and $b_3v_2+c_3v_3$ are linearly dependent,
 so we have $$\phi(u_1,u_2,u_3)\left(
\begin{array}{cc}
1 &  \\
 & B
\end{array}
\right)
=(v_1,v_2,v_3)\left(
\begin{array}{ccc}
0 &0 &1  \\
b_1 &b'_2 &0  \\
c_1 &c'_2 &0 
\end{array}
\right)$$
 for some $B\in GL(2,\Z)$.
 Putting $v'_2:=b_1v_2+c_1v_3$ and $v'_3:=b'_2v_2+c'_2v_3$, we see that $\{v'_2,v'_3\}$ is a basis of $H^1(N_2)$ and $v'_2v'_3=0$.
 Since quasitoric manifolds over $C^3(5)^*$ are cohomologically rigid, $N_2$ is homeomorphic to $\C P^3 \# \C P^3$,
 and therefore $N$ is homeomorphic to $\C P^3 \# \C P^3 \# \C P^3$.
 Similarly, taking $(u'_2,u'_3)=\phi^{-1}(v'_3,v_1)\in H^2(M_2)$, $u'_2u'_3=0$ so $M$ is also homeomorphic to $\C P^3 \# \C P^3 \# \C P^3$.
 \item If $a_1\neq 0$, we have $a_2=a_3=0$ similarly.
 Then we see that the correspondence $u_2\mapsto b_2v_2+c_2v_3,u_3\mapsto b_3v_2+c_3v_3$ provides an isomorphism $H^*(M_2)\rightarrow H^*(N_2)$.
 Therefore, from the cohomological rigidity of quasitoric manifolds over $C^3(5)^*$, $M_2$ is homeomorphic to $N_2$.
 Since $M_1=N_1=\C P^3$, $M$ and $N$ are also homeomorphic.
\end{itemize}
Then we see that $M$ is homeomorphic to $N$ in each case.
\end{proof}

We obtain the following theorem.

\begin{thm}\label{thm:qtmfd over C3(6) up to homeo}
Over $C^3(6)^*$, there exist countably infinite quasitoric manifolds up to homeomorphism and they are distinguished by their integral cohomology rings.
\end{thm}

The classification of quasitoric manifolds over $C^3(6)^*$ is comparatively easy
since indecomposable characteristic matrices are almost parametrized by one integer $d$,
but it is more complicated in the cases of $C^3(m)^*$ for $m>6$.
There are countably infinite indecomposable characteristic matrices on $C^3(m)^*$ up to equivalence, and for large $m$, it seems hard to list them all.
Then, for the present, let us consider the classification of small covers over $C^3(m)^*$.

\begin{lem}
Any real characteristic matrix on $C^3(m)^*$ for $m>5$ is decomposable.
\end{lem}
\begin{proof}
Let $\lambda$ be a real characteristic matrix on $C^3(m)^*$ for $m>5$.
We can assume that $\lambda$ is in the form
$$\lambda=\left(
\begin{array}{ccccccc}
1 &0 &0 &a_1 &b_1 &\cdots &1 \\
0 &1 &0 &1 &b_2 &\cdots &c_2  \\
0 &0 &1 &a_3 &b_3 &\cdots &1 
\end{array}
\right),$$
since $\minor{1,2,3}=\minor{1,3,4}=\minor{1,2,m}=\minor{2,3,m}=1$.
If $c_2=1$, then $\minor{1,3,m}=1$, so $\lambda$ is decomposable.
Otherwise, if $c_2=0$, then $\minor{1,4,m}=1$. Hence $\lambda$ is decomposable.
\end{proof}

Then we list all indecomposable real characteristic matrices over $C^3(m)^*$ for $m=4,5$ up to equivalence.
\begin{itemize}
 \item On $C^3(4)^*$, there exists only one real characteristic matrix
 $$\left(
\begin{array}{cccc}
1 &0 &0 &1  \\
0 &1 &0 &1  \\
0 &0 &1 &1 
\end{array}
\right)$$
up to equivalence, and the corresponding small cover is $\R P^3$.
 \item On $C^3(5)^*=\Delta^1\times \Delta^2$, there exist two indecomposable real characteristic matrices up to equivalence:
 $$\lambda_1=\left(
\begin{array}{cc|ccc}
1 &1 &0 &0 &0  \\ \hline
0 &0 &1 &0 &1  \\
0 &0 &0 &1 &1 
\end{array}
\right),
\lambda_2=\left(
\begin{array}{cc|ccc}
1 &1 &0 &0 &0  \\ \hline
0 &1 &1 &0 &1  \\
0 &1 &0 &1 &1 
\end{array}
\right),$$
where we label the facets $F_1,\ldots,F_5$ such that $F_1,F_2$ correspond to the facets of $\Delta^1$, and $F_3,F_4,F_5$ to those of $\Delta^2$.
\end{itemize}

Using the real moment-angle manifold ${}_\R \mathcal{Z}_{\Delta^1\times \Delta^2}=S^1\times S^2$,
we easily see that ${}_\R M(\lambda_1)$ is homeomorphic to $\R P^1\times \R P^2$.
Moreover, we can construct a homeomorphism from ${}_\R M(\lambda_2)$ to ${}_\R M(\lambda_1)$ explicitly, as the induced map from
$$S^1\times S^2\rightarrow S^1\times S^2\colon (z_1,z_2,w_1,w_2,w_3)\mapsto (z_1,z_2,z_1w_1+z_2w_2,-z_2w_1+z_1w_2,w_3),$$
which is a weakly equivariant homeomorphism with respect to the free action of $S_{\lambda_2}$ on the domain,
and that of $S_{\lambda_1}$ on the codomain.
Here each $S_{\lambda_i}$ denotes the kernel of $\lambda_i\colon\ZZ^5\rightarrow \ZZ^3$ for $i=1,2$.
This construction follows the one for quasitoric manifolds over $\Delta^1\times \Delta^n$ in \cite{CPS12}.

Then we obtain the topological classification of small covers over $C^3(m)^*$.

\begin{thm}\label{thm:smlcvr over C3(6) up to homeo}
For any $m>3$, a small cover over $C^3(m)^*$ is homeomorphic to the connected sum of copies of $\R P^3$ and $\R P^1\times \R P^2$.
\end{thm}

%
%
%

\section{SMALL COVERS OF HIGHER DIMENSIONS}
In this section, we list all real characteristic matrices over the dual cyclic polytopes of dimension $n \geq 4$.
As is noticed in Section 1, we only consider the case that $m-n\geq 3$, where $m$ denotes the number of facets.
We will use the easy lemma below.

\begin{lem}\label{lem:(n,m) to (n+1,m+1)}
If there exists a (real) characteristic matrix $\lambda$ on $C^{n+1}(m+1)^*$ in the form
$$
\lambda=
\left(
\begin{array}{c|ccc}
1&a_2&\cdots&a_m\\\hline
0&&&\\
\vdots&\multicolumn{3}{c}{\raisebox{-5pt}[0pt][0pt]{\Huge $A$}}\\
0&&\\
\end{array}
\right),
$$
then $A$ is a (real) characteristic matrix on $C^n(m)^*$.
\end{lem}
\begin{proof}
It follows immediately from Gale's evenness condition:
if $\{ i_1,\ldots,i_n \}$ is a simplex of $K_{C^n(m)^*}$,
then $\{1,i_1+1,\ldots,i_n+1\}$ is a simplex of $K_{C^{n+1}(m+1)^*}$.
Hence we have $$\vert i_1,\ldots,i_n \vert_A=\vert 1,i_1+1,\ldots,i_n+1\vert_{\lambda}=\pm 1$$ for any $\{ i_1,\ldots,i_n \}\in K_{C^n(m)^*}$.
\end{proof}

\begin{prop}\label{prop:no existence m-n>3}
If $n\geq 4$ and $m-n\geq 4$, there exists no real characteristic matrix on $C^n(m)^*$.
\end{prop}
\begin{proof}
By the previous lemma, we only have to prove that there exists no real characteristic matrix on $C^4(m)^*$ if $m\geq 8$.
Assume that there exists a real characteristic matrix $\lambda$ on $C^4(m)^*$ for $m\geq 8$.
All maximal faces of $K_{C^4(m)^*}$ are listed below:
$$ \{1,i,i+1,m\}\,(i=2,3,\ldots,m-2),\,\{i,i+1,j,j+1\}\,(1\leq i<j-1\leq m-2). $$
We can assume that $\lambda$ is in the form
$$
\left(
\begin{array}{cccc|ccccc}
1 &0 &0 &0 &1 &b_1 &c_1 &\cdots &d_1  \\
0 &1 &0 &0 &a_2 &b_2 &c_2 &\cdots &1  \\
0 &0 &1 &0 &1 &b_3 &c_3 &\cdots &d_3  \\
0 &0 &0 &1 &a_4 &b_4 &c_4 &\cdots &1 
\end{array}
\right).
$$
Let us consider four cases (1) $\mathbf{a}=(1,0,1,0)$, (2) $\mathbf{a}=(1,1,1,1)$, (3) $\mathbf{a}=(1,1,1,0)$, and (4) $\mathbf{a}=(1,0,1,1)$ separately,
where $\mathbf{a}:=(a_1,a_2,a_3,a_4)$.

\begin{itemize}
 \item \textbf{Case(1)}: $\mathbf{a}=(1,0,1,0)$.
 
 In this case, $\lambda$ is in the form
$$
\left(
\begin{array}{cccc|ccccc}
1 &0 &0 &0 &1 &b_1 &c_1 &\cdots &d_1  \\
0 &1 &0 &0 &0 &b_2 &c_2 &\cdots &1  \\
0 &0 &1 &0 &1 &b_3 &c_3 &\cdots &d_3  \\
0 &0 &0 &1 &0 &b_4 &c_4 &\cdots &1 
\end{array}
\right).
$$
From the non-singular condition, we have
$$1=\minor{1,2,5,6}=b_4,\,1=\minor{3,4,5,6}=b_2.$$
Then we have $\minor{1,5,6,m}=0$, but this contradicts the non-singular condition.
 \item \textbf{Case(2)}: $\mathbf{a}=(1,1,1,1)$.

$\lambda$ is in the form
$$
\left(
\begin{array}{cccc|ccccc}
1 &0 &0 &0 &1 &b_1 &c_1 &\cdots &d_1  \\
0 &1 &0 &0 &1 &b_2 &c_2 &\cdots &1  \\
0 &0 &1 &0 &1 &b_3 &c_3 &\cdots &d_3  \\
0 &0 &0 &1 &1 &b_4 &c_4 &\cdots &1 
\end{array}
\right).
$$
From the non-singular condition, we have
$$1=\minor{1,2,5,6}=b_3+b_4,\,1=\minor{2,3,5,6}=b_1+b_4,\,1=\minor{3,4,5,6}=b_1+b_2.$$
Therefore $\mathbf{b}:=(b_1,b_2,b_3,b_4)$ corresponds with $(1,0,1,0)$ or $(0,1,0,1)$.
In each case, we have the contradiction $1=\minor{1,5,6,m}=0$.
 \item \textbf{Case(3)}: $\mathbf{a}=(1,1,1,0)$.

$\lambda$ is in the form
$$
\left(
\begin{array}{cccc|ccccc}
1 &0 &0 &0 &1 &b_1 &c_1 &\cdots &d_1  \\
0 &1 &0 &0 &1 &b_2 &c_2 &\cdots &1  \\
0 &0 &1 &0 &1 &b_3 &c_3 &\cdots &d_3  \\
0 &0 &0 &1 &0 &b_4 &c_4 &\cdots &1 
\end{array}
\right).
$$
From the non-singular condition, we have
$b_4=1$, $b_1+b_2=1$, $d_3=0$ and $b_2+b_3=0$.
Therefore $\mathbf{b}$ corresponds with $(0,1,1,1)$ or $(1,0,0,1)$.
\begin{itemize}
 \item In the case that $\mathbf{b}=(0,1,1,1)$, we have
 $1=\minor{1,2,6,7}=c_3+c_4$, $1=\minor{2,3,6,7}=c_1$ and $1=\minor{4,5,6,7}=c_2+c_3$.
 In particular, we have $c_2=c_4$. Then there occurs a contradiction $1=\minor{1,6,7,m}=0$.
 \item In the case that $\mathbf{b}=(1,0,0,1)$, we have
 $1=\minor{1,2,6,7}=c_3,\,1=\minor{3,4,6,7}=c_2$, and then $1=\minor{4,5,6,7}=0$.
This is a contradiction.
\end{itemize}
 \item \textbf{Case(4)}: $\mathbf{a}=(1,0,1,1)$.

$\lambda$ is in the form
$$
\left(
\begin{array}{cccc|ccccc}
1 &0 &0 &0 &1 &b_1 &c_1 &\cdots &d_1  \\
0 &1 &0 &0 &0 &b_2 &c_2 &\cdots &1  \\
0 &0 &1 &0 &1 &b_3 &c_3 &\cdots &d_3  \\
0 &0 &0 &1 &1 &b_4 &c_4 &\cdots &1 
\end{array}
\right).
$$
From the non-singular condition, we have
$b_3+b_4=1,\,b_1+b_4=1,\,b_2=1$, and then $d_3=1$.
Therefore $\mathbf{b}$ corresponds with $(0,1,0,1)$ or $(1,1,1,0)$.
\begin{itemize}
 \item In the case that $\mathbf{b}=(0,1,0,1)$, we have $1=\minor{1,2,6,7}=c_3,\ 1=\minor{2,3,6,7}=c_1$, and then $\minor{4,5,6,7}=0.$
This is a contradiction.
 \item In the case that $\mathbf{b}=(1,1,1,0)$, we have
 $1=\minor{1,2,6,7}=c_4$, $1=\minor{3,4,6,7}=c_1+c_2$ and $1=\minor{4,5,6,7}=c_1+c_3$.
In particular, we have $c_2=c_3$. However, this is a contradiction since $\minor{1,6,7,m}=0$.
\end{itemize}
\end{itemize}

Thus the proof is completed.
\end{proof}

Next, we will determine all real characteristic matrices on $C^4(7)^*$.
Note that maximal faces of $K_{C^4(7)^*}$ are as below:
$$\{1,2,3,4\},\{2,3,4,5\},\{1,2,4,5\},\{1,3,4,7\},\{1,2,3,7\},\{1,2,5,6\},\{1,2,6.7\},$$
$$\{2,3,5,6\},\{2,3,6,7\},\{3,4,5,6\},\{3,4,6,7\},\{1,4,5,7\},\{1,5,6,7\},\{4,5,6,7\}.$$
\begin{prop}
$\XX_{C^4(7)^*}$ has exactly two elements, and they are respectively represented by the matrices below:
$$\left(
\begin{array}{cccc|ccc}
1 &0 &0 &0 &1 &0 &1  \\
0 &1 &0 &0 &1 &1 &1  \\
0 &0 &1 &0 &1 &1 &0  \\
0 &0 &0 &1 &0 &1 &1 
\end{array}
\right),
\left(
\begin{array}{cccc|ccc}
1 &0 &0 &0 &1 &1 &0  \\
0 &1 &0 &0 &0 &1 &1  \\
0 &0 &1 &0 &1 &1 &1  \\
0 &0 &0 &1 &1 &0 &1 
\end{array}
\right).
$$
\end{prop}
\begin{proof}
It is straightforward to check that the above two matrices are real characteristic matrices on $C^4(7)^*$.
Then we prove that there exists no other real characteristic matrix on $C^4(7)^*$ up to the action of $GL(4,\Z/2)$.
As in the proof of the previous proposition, we can write a real characteristic matrix on $C^4(7)^*$ in the form:
$$\lambda=
\left(
\begin{array}{cccc|ccc}
1 &0 &0 &0 &1 &b_1 &c_1 \\
0 &1 &0 &0 &a_2 &b_2 &1 \\
0 &0 &1 &0 &1 &b_3 &c_3 \\
0 &0 &0 &1 &a_4 &b_4 &1 
\end{array}
\right).
$$
We consider the cases (1) $\mathbf{a}=(1,0,1,0)$, (2) $\mathbf{a}=(1,1,1,1)$, (3) $\mathbf{a}=(1,1,1,0)$, and (4) $\mathbf{a}=(1,0,1,1)$ separately,
where $\mathbf{a}:=(1,a_2,1,a_4)$.
In the former two cases, we have contradictions by replacing $m$ with 7 in the previous proof.
\begin{itemize}
 \item \textbf{Case(3)}: $\mathbf{a}=(1,1,1,0)$.

$\lambda$ is in the form
$$
\left(
\begin{array}{cccc|ccc}
1 &0 &0 &0 &1 &b_1 &c_1  \\
0 &1 &0 &0 &1 &b_2 &1  \\
0 &0 &1 &0 &1 &b_3 &c_3  \\
0 &0 &0 &1 &0 &b_4 &1 
\end{array}
\right).
$$
From the non-singular condition, we have
$$1=\minor{1,2,5,6}=b_4,\,1=\minor{3,4,5,6}=b_1+b_2,\,1=\minor{1,4,5,7}=c_3+1,$$
$$1=\minor{1,5,6,7}=b_2+b_3+1,\,1=\minor{1,2,6,7}=b_3.$$
Therefore $\mathbf{b}:=(b_1,b_2,b_3,b_4)$ corresponds with $(0,1,1,1)$.
We also have $c_1=\minor{2,3,6,7}=1$,
so $\lambda$ is equal to the left matrix in the statement.
 \item \textbf{Case(4)}: $\mathbf{a}=(1,0,1,1)$.

$\lambda$ is in the form
$$
\left(
\begin{array}{cccc|ccc}
1 &0 &0 &0 &1 &b_1 &c_1  \\
0 &1 &0 &0 &0 &b_2 &1  \\
0 &0 &1 &0 &1 &b_3 &c_3  \\
0 &0 &0 &1 &1 &b_4 &1 
\end{array}
\right).
$$
From the non-singular condition, we have
$$1=\minor{1,2,5,6}=b_3+b_4,\,1=\minor{2,3,5,6}=b_1+b_4,\,1=\minor{3,4,5,6}=b_2,$$
$$1=\minor{1,5,6,7}=c_3,\,
1=\minor{4,5,6,7}=c_1+1.$$
Therefore $\mathbf{c}:=(c_1,c_2,c_3,c_4)$ corresponds with $(0,1,1,1)$.
We also have $b_1=\minor{3,4,6,7}=1$,
so $\lambda$ is equal to the right matrix in the statement.
\end{itemize}
Thus the proof is completed.
\end{proof}

\begin{cor}
$\XX_{C^5(8)^*}$ has exactly two elements, and they are respectively represented by the matrices below:
$$\left(
\begin{array}{ccccc|ccc}
1 &0 &0 &0 &0 &0 &0 &1  \\
0 &1 &0 &0 &0 &1 &0 &1  \\
0 &0 &1 &0 &0 &1 &1 &1  \\
0 &0 &0 &1 &0 &1 &1 &0  \\
0 &0 &0 &0 &1 &0 &1 &1 
\end{array}
\right),
\left(
\begin{array}{ccccc|ccc}
1 &0 &0 &0 &0 &0 &0 &1  \\
0 &1 &0 &0 &0 &1 &1 &0  \\
0 &0 &1 &0 &0 &0 &1 &1  \\
0 &0 &0 &1 &0 &1 &1 &1  \\
0 &0 &0 &0 &1 &1 &0 &1 
\end{array}
\right).
$$
\end{cor}
\begin{proof}
We can easily check that the above two matrices satisfy the non-singular condition for $C^5(8)^*$.
Let $\lambda$ be a real characteristic matrix on $C^5(8)^*$.
By Lemma \ref{lem:(n,m) to (n+1,m+1)}, we can assume that $\lambda$ is in the form:
$$\left(
\begin{array}{c|cccc|ccc}
1 &0 &0 &0 &0 &a &b &1  \\\hline
0 &1 &0 &0 &0 &1 &0 &1  \\
0 &0 &1 &0 &0 &1 &1 &1  \\
0 &0 &0 &1 &0 &1 &1 &0  \\
0 &0 &0 &0 &1 &0 &1 &1 
\end{array}
\right)\ \mathrm{or}\ 
\left(
\begin{array}{c|cccc|ccc}
1 &0 &0 &0 &0 &a &b &1  \\\hline
0 &1 &0 &0 &0 &1 &1 &0  \\
0 &0 &1 &0 &0 &0 &1 &1  \\
0 &0 &0 &1 &0 &1 &1 &1  \\
0 &0 &0 &0 &1 &1 &0 &1 
\end{array}
\right).
$$

In the former case, the non-singular condition provides
 $$1=\minor{2,3,6,7,8}=a+b+1,\ 1=\minor{4,5,6,7,8}=a+1.$$
Hence $\lambda$ is equal to the left matrix in the statement.

In the latter case, the non-singular condition provides
 $$1=\minor{2,3,6,7,8}=a+1,\ 1=\minor{3,4,6,7,8}=a+b+1.$$
Hence $\lambda$ is equal to the right matrix in the statement.
\end{proof}

For any matrix in the form
$$\left(\begin{array}{c|ccccc|ccc}
1 &0 &0 &0 &0 &0 &1 &a &b  \\\hline
0 &1 &0 &0 &0 &0 &0 &0 &1  \\
0 &0 &1 &0 &0 &0 &1 &0 &1  \\
0 &0 &0 &1 &0 &0 &1 &1 &1  \\
0 &0 &0 &0 &1 &0 &1 &1 &0  \\
0 &0 &0 &0 &0 &1 &0 &1 &1 
\end{array}
\right)\,\mathrm{or}\,
\left(
\begin{array}{c|ccccc|ccc}
1 &0 &0 &0 &0 &0 &1 &a &b  \\\hline
0 &1 &0 &0 &0 &0 &0 &0 &1  \\
0 &0 &1 &0 &0 &0 &1 &1 &0  \\
0 &0 &0 &1 &0 &0 &0 &1 &1  \\
0 &0 &0 &0 &1 &0 &1 &1 &1  \\
0 &0 &0 &0 &0 &1 &1 &0 &1 
\end{array}
\right),
$$
it is obvious that $\minor{3,4,5,6,7,8}=0$.
Then we have the following corollary.

\begin{cor}\label{cor:no existence n>5}
There exists no real characteristic matrix on $C^n(n+3)^*$ for $n>5$.
\end{cor}


%
%
%

\section{CLASSIFICATION FOR $C^4(7)^*$}
In the previous section, we showed that there are exactly two real characteristic matrices on $C^4(7)^*$ up to the left $GL(4,\Z/2)$-action,
and they are represented by the matrices $\overline{\lambda}$ and $\overline{\lambda'}$ below:
$$\overline{\lambda}=\left(
\begin{array}{cccc|ccc}
1 &0 &0 &0 &1 &0 &1  \\
0 &1 &0 &0 &1 &1 &1  \\
0 &0 &1 &0 &1 &1 &0  \\
0 &0 &0 &1 &0 &1 &1 
\end{array}
\right),\,
\overline{\lambda'}=\left(
\begin{array}{cccc|ccc}
1 &0 &0 &0 &1 &1 &0  \\
0 &1 &0 &0 &0 &1 &1  \\
0 &0 &1 &0 &1 &1 &1  \\
0 &0 &0 &1 &1 &0 &1 
\end{array}
\right).$$
In the way similar to the proof of Lemma \ref{lem:Aut on C3(m)},
it is shown that $\mathrm{Aut}\,(K_{C^4(7)^*})$ is generated by $\sigma$ and $\tau$ below:
$$\sigma=\left(
\begin{array}{ccccccc}
1 &2 &3 &4 &5 &6 &7  \\
2 &3 &4 &5 &6 &7 &1 
\end{array}
\right),\,
\tau=\left(
\begin{array}{ccccccc}
1 &2 &3 &4 &5 &6 &7  \\
7 &6 &5 &4 &3 &2 &1 
\end{array}
\right).$$
Note that $\mathrm{Aut}\,(K_{C^4(7)^*})$ acts on $\XX_{C^4(7)^*}$ by
$\sigma(\overline{\lambda})=\overline{\lambda}$, $\sigma(\overline{\lambda'})=\overline{\lambda'}$ and $\tau(\overline{\lambda})=\overline{\lambda'}$.
Then we have the following proposition.

\begin{prop}\label{prop:smlcvr over C4(7)}
There exists only one small cover over $C^4(7)^*$ up to weakly equivariant homeomorphism.
\end{prop}

Let $p$ denote the map $\X_{C^4(7)^*}\rightarrow \XX_{C^4(7)^*}$ induced from the modulo 2 reduction.
Then we see that $p^{-1}(\overline{\lambda})/\langle \sigma \rangle$ corresponds with $\X_{C^4(7)^*}/\mathrm{Aut}\,(K_{C^4(7)^*})$,
so first we have to list all elements of $p^{-1}(\overline{\lambda})$.
By a direct calculation,
we see that there are exactly twenty-eight elements in $p^{-1}(\overline{\lambda})$, which are represented by the characteristic matrices below:
\begin{align*}
\lambda_1 &=\chmt{0}{-1}{-1}{1}{1}{0}{0}{1}, &\lambda_2 &=\chmt{2}{1}{1}{1}{1}{0}{0}{1},&\lambda_3 &=\chmt{0}{1}{1}{1}{1}{0}{2}{1}, \\
\lambda_4 &=\chmt{2}{1}{1}{3}{1}{0}{2}{3},&\lambda_5 &=\chmt{0}{1}{1}{1}{1}{0}{0}{-1}, &\lambda_6 &=\chmt{2}{1}{1}{3}{1}{0}{0}{1}, \\
\lambda_7 &=\chmt{0}{1}{1}{-1}{1}{0}{2}{1}, &\lambda_8 &=\chmt{2}{1}{1}{1}{1}{0}{2}{3},&\lambda_9 &=\chmt{0}{1}{1}{1}{1}{0}{0}{1}, \\
\lambda_{10} &=\chmt{-2}{-1}{-1}{1}{1}{0}{0}{1},&\lambda_{11} &=\chmt{0}{-1}{1}{1}{1}{0}{0}{1}, &\lambda_{12} &=\chmt{0}{1}{3}{1}{1}{0}{2}{1}, \\
\lambda_{13} &=\chmt{0}{1}{3}{1}{1}{0}{0}{-1}, &\lambda_{14} &=\chmt{0}{1}{-1}{-1}{1}{0}{2}{1},&\lambda_{15} &=\chmt{0}{-1}{1}{1}{1}{0}{0}{-1}, \\
\lambda_{16} &=\chmt{0}{-1}{1}{1}{1}{0}{2}{1},&\lambda_{17} &=\chmt{0}{-1}{-1}{-1}{1}{0}{2}{1}, &\lambda_{18} &=\chmt{2}{3}{1}{1}{1}{0}{0}{1}, \\
\lambda_{19} &=\chmt{0}{1}{1}{1}{3}{2}{0}{1}, &\lambda_{20} &=\chmt{2}{1}{1}{1}{1}{2}{0}{1},&\lambda_{21} &=\chmt{0}{1}{1}{1}{1}{2}{0}{1}, \\
\lambda_{22} &=\chmt{2}{1}{1}{1}{3}{2}{0}{1},&\lambda_{23} &=\chmt{0}{1}{1}{-1}{1}{2}{0}{1}, &\lambda_{24} &=\chmt{2}{1}{1}{3}{3}{2}{0}{1}, \\
\lambda_{25} &=\chmt{0}{1}{1}{1}{1}{2}{2}{1}, &\lambda_{26} &=\chmt{0}{-1}{1}{1}{1}{2}{0}{1},&\lambda_{27} &=\chmt{0}{-1}{1}{1}{1}{2}{2}{1}, \\
\lambda_{28} &=\chmt{2}{3}{1}{1}{1}{2}{0}{1},
\end{align*}
where we mean by $*$ the whole matrix $(I\,\vert\,*)$ as in Proposition \ref{prop:list(3,6)}.

Next, we consider the action of $\sigma$ on $p^{-1}(\overline{\lambda})$.
By a direct calculation, we see that the action is illustrated as follows.
\[\xymatrix{
\lambda_1\ar[r]^{\sigma}&\lambda_9\ar[r]^{\sigma}&\lambda_5\ar[r]^{\sigma}&\lambda_{21}\ar[r]^{\sigma}&
\lambda_3\ar[r]^{\sigma}&\lambda_2\ar[r]^{\sigma}&\lambda_{11}\ar[r]^{\sigma}&\lambda_1
}\]
\[\xymatrix{
\lambda_4\ar[r]^{\sigma}&\lambda_{18}\ar[r]^{\sigma}&\lambda_{13}\ar[r]^{\sigma}&\lambda_{19}\ar[r]^{\sigma}&
\lambda_{17}\ar[r]^{\sigma}&\lambda_{14}\ar[r]^{\sigma}&\lambda_{27}\ar[r]^{\sigma}&\lambda_4
}\]
\[\xymatrix{
\lambda_6\ar[r]^{\sigma}&\lambda_{26}\ar[r]^{\sigma}&\lambda_7\ar[r]^{\sigma}&\lambda_{20}\ar[r]^{\sigma}&
\lambda_{14}\ar[r]^{\sigma}&\lambda_{22}\ar[r]^{\sigma}&\lambda_{16}\ar[r]^{\sigma}&\lambda_6
}\]
\[\xymatrix{
\lambda_8\ar[r]^{\sigma}&\lambda_{28}\ar[r]^{\sigma}&\lambda_{12}\ar[r]^{\sigma}&\lambda_{10}\ar[r]^{\sigma}&
\lambda_{15}\ar[r]^{\sigma}&\lambda_{23}\ar[r]^{\sigma}&\lambda_{25}\ar[r]^{\sigma}&\lambda_8
}\]
Then we obtain the following proposition.

\begin{prop}\label{prop:qtmfd over C4(7) up to w.e.homeo}
Over $C^4(7)^*$, there are exactly four quasitoric manifolds $M(\lambda_9)$, $M(\lambda_{17})$, $M(\lambda_{16})$ and $M(\lambda_{15})$
up to weakly equivariant homeomorphism.
\end{prop}

Next, putting $A:=H^*(\lambda_9)$, $B:=H^*(\lambda_{17})$, $C:=H^*(\lambda_{16})$, $D:=H^*(\lambda_{15})$,
let us show that $A,B,C,D$ are not isomorphic to each other as graded rings.
With the notations in Theorem \ref{thm:coh,quasi}, we put $X:=v_5,Y:=v_6,Z:=v_7$.
We can take $[M]:=v_1v_5v_6v_7$ as the fundamental cohomology class.
Indeed, each $v_i$ is the Poincar\'{e} dual of $\pi^{-1}(F_i)$ and
$\pi^{-1}(F_1)$, $\pi^{-1}(F_5)$, $\pi^{-1}(F_6)$, $\pi^{-1}(F_7)$ intersect transversally at the point $\pi^{-1}(v)$,
where $v$ denotes the vertex $F_1\cap F_5\cap F_6\cap F_7$.
Table 1 shows the coefficient of $[M]$ in each monomial.
\begin{table}[h]
\begin{center}
$
\begin{array}{|c||c|c|c|c|c|c|c|c|} \hline
 &X^4 &Y^4 &Z^4 &X^3Y &X^2Y^2 &XY^3 &Y^3Z &Y^2Z^2  \\ \hhline{|=::========|}
A &0 &0 &-2 &0 &0 &1 &0 &0  \\ \hline
B &0 &-4 &0 &0 &0 &1 &2 &0  \\ \hline
C &0 &6 &-2 &0 &0 &-1 &0 &0  \\ \hline
D &0 &-2 &0 &0 &0 &1 &2 &0  \\ \hline \hhline{:========~}
 &YZ^3 &Z^3X &Z^2X^2 &ZX^3 &X^2YZ &XY^2Z &XYZ^2 \\ \hhline{|=::=======|}
A &1 &0 &0 &-1 &1 &-1 &0  \\ \hhline{|--------|}
B &1 &0 &0 &-1 &1 &-1 &0  \\ \hhline{|--------|}
C &1 &0 &0 &-1 &1 &-1 &0  \\ \hhline{|--------|}
D &1 &0 &0 &-1 &1 &-1 &0  \\ \hhline{|--------|}
\end{array}
$
\caption{}
\end{center}
\end{table}

\begin{lem}
$B$ is not isomorphic to $A$, $C$, and $D$.
\end{lem}
\begin{proof}
If we assume that there exists a graded ring isomorphism $\phi\colon A\rightarrow B$,
we have $\pm 2[M]=\phi(Z^4)=\phi(Z)^4=(aX+bY+cZ)^4$ for some integers $a,b,c$.
However, since $X^4=Y^4=Z^4=X^2Y^2=Y^2Z^2=Z^2X^2=0$ in $B$, we see that $(aX+bY+cZ)^4\equiv 0 \bmod 4$.
This is a contradiction.
We can prove the lemma similarly for $C$ and $D$.
\end{proof}

\begin{lem}
$A$ and $C$ are not isomorphic to $D$.
\end{lem}
\begin{proof}
It suffices to show that there exists no graded ring isomorphism from $A/(3)$ and $C/(3)$ to $D/(3)$.
Assume that there exists a graded ring isomorphism $\phi\colon A/(3)\rightarrow D/(3)$,
and take vectors $v_i=(a_i,b_i,c_i)\in (\Z/3)^3$ for $i=1,2,3$ such that $$\left(
\begin{array}{c}
\phi(X) \\
\phi(Y) \\
\phi(Z) 
\end{array}
\right)=\left(
\begin{array}{ccc}
a_1 &b_1 &c_1  \\
a_2 &b_2 &c_2  \\
a_3 &b_3 &c_3 
\end{array}
\right)\left(
\begin{array}{c}
X \\
Y \\
Z 
\end{array}
\right).$$
From $\phi(X)^4=0$, we obtain the equation $2a_1c_1+a_1b_1+b_1^2=0$.
Similarly, from $\phi(Y)^4=0$, we obtain $2a_2c_2+a_2b_2+b_2^2=0$.
From a direct calculation, we see that $$v_i= \pm(0,0,1),\pm(1,0,0),\pm(1,1,2),\pm(1,2,0)$$  for $i=1,2$.
\begin{itemize}
 \item If $v_1=\pm (0,0,1)$, from $\phi(X^3Y)=0$ and the list above, we have $v_2=\pm (1,0,0)$.
 Then, from $\phi(Z^3X)=\phi(Y^3Z)=0$, we obtain $v_3=\pm (1,2,0)$,
 but this is a contradiction since $0\neq \phi(Z^4)=\pm (X+2Y)^4=0$.
 \item If $v_1=\pm (1,0,0)$, from $\phi(X^3Y)=0$ and the list above, we have $v_2=\pm (1,2,0)$.
 Then, from $\phi(Z^3X)=\phi(Y^3Z)=0$, we obtain $v_3=\pm (0,0,1)$,
 but this is a contradiction since $0\neq \phi(Z^4)=\pm Z^4=0$.
 \item If $v_1=\pm (1,1,2)$, from $\phi(X^3Y)=0$ and the list above, we have $v_2=\pm (1,1,2)$,
 but this contradicts the linear independence of $v_1$ and $v_2$.
 \item If $v_1=\pm (1,2,0)$, from $\phi(X^3Y)=0$ and the list above, we have $v_2=\pm (0,0,1)$.
 Then, from $\phi(Z^3X)=\phi(Y^3Z)=0$, we obtain $v_3=\pm (1,0,0)$,
 but this is a contradiction since $0\neq \phi(Z^4)=\pm X^4=0$.
\end{itemize}

This proof also works when we consider $C$ instead of $A$.
\end{proof}

\begin{lem}
$A$ and $C$ are not isomorphic as graded rings.
\end{lem}
\begin{proof}
We prove that there exists no graded ring isomorphism from $A/(3)$ to $C/(3)$.
Let us assume that there exists a graded ring isomorphism $\phi\colon A/(3)\rightarrow C/(3)$.
With the same notations as the proof of the previous lemma, we obtain the equations $c_i^2+2a_ib_i+b_ic_i+2a_ic_i=0$ for $i=1,2$ from $\phi(X^4)=\phi(Y^4)=0$.
Then we see that $$v_i= \pm(0,1,0),\pm(0,1,2),\pm(1,0,0),\pm(1,0,1),\pm(1,1,1),\pm(1,1,2),\pm(1,2,1)$$ for $i=1,2$.
\begin{itemize}
 \item If $v_1=\pm (0,1,0)$, from $\phi(X^3Y)=0$ and the list above, we have $v_2=\pm (0,1,2)$.
 Then, from $\phi(Z^3X)=\phi(Y^3Z)=0$, we obtain $v_3=\pm (1,2,0)$,
 but this is a contradiction since $0= \phi(Z^2X^2)\neq 0$.
 This proof also works in the case that $v_1=\pm (1,1,1)$.
 \item If $v_1=\pm (0,1,2)$, from $\phi(X^3Y)=0$ and the list above, we have $v_2=\pm (1,1,1)$.
 Then there occurs a contradiction that $0=\phi(X^2Y^2)\neq 0$.
 This proof also works in the cases that $v_1=\pm (1,0,1),\ \pm (1,1,2)$.
 \item If $v_1=\pm (1,0,0)$, from $\phi(X^3Y)=0$ and the list above, we have $v_2=(0,b,0)$ for $b=\pm 1$.
 Then, from $\phi(Z^3X)=\phi(Y^3Z)=0$, we obtain $v_3=\pm (0,0,c)$ for $c=\pm 1$.
 Besides, from $Z^4=XY^3=YZ^3=2ZX^3$ in $A/(3)$, we have $1=2b=bc=c$.
 However, this equation has no root.
 \item If $v_1=\pm (1,2,1)$, from $\phi(X^3Y)=0$ and the list above, we have $v_2=\pm (1,2,1)$.
 This contradicts the linear independence of $v_1$ and $v_2$.
\end{itemize}
Hence we see that there exists no isomorphism $A/(3)\rightarrow C/(3)$.
\end{proof}

Thus we complete the topological classification of quasitoric manifolds over $C^4(7)^*$.

\begin{thm}\label{thm:qtmfd over C4(7) up to homeo}
Any quasitoric manifold over $C^4(7)^*$ is homeomorphic to $M(\lambda_9)$, $M(\lambda_{17})$, $M(\lambda_{16})$ or $M(\lambda_{15})$,
and their cohomology rings are not isomorphic.
\end{thm}

%
%
%

\section{CLASSIFICATION FOR $C^5(8)^*$}
In Section 6, we showed that $\XX_{C^5(8)^*}$ has exactly two elements and they are represented by the real characteristic matrices below:
$$\left(
\begin{array}{ccccc|ccc}
1 &0 &0 &0 &0 &0 &0 &1  \\
0 &1 &0 &0 &0 &1 &0 &1  \\
0 &0 &1 &0 &0 &1 &1 &1  \\
0 &0 &0 &1 &0 &1 &1 &0  \\
0 &0 &0 &0 &1 &0 &1 &1 
\end{array}
\right),
\left(
\begin{array}{ccccc|ccc}
1 &0 &0 &0 &0 &0 &0 &1  \\
0 &1 &0 &0 &0 &1 &1 &0  \\
0 &0 &1 &0 &0 &0 &1 &1  \\
0 &0 &0 &1 &0 &1 &1 &1  \\
0 &0 &0 &0 &1 &1 &0 &1 
\end{array}
\right).$$
As in the previous section, we denote these matrices by $\overline{\lambda}$, $\overline{\lambda'}$
and the map $\X_{C^5(8)^*}\rightarrow \XX_{C^5(8)^*}$ by $p$.
It is shown in the same way as Lemma \ref{lem:Aut on C3(m)} that $\mathrm{Aut}\,(K_{C^5(8)^*})$ is generated by $\sigma$ and $\tau$ below:
$$\sigma=\left(
\begin{array}{cccccccc}
1 &2 &3 &4 &5 &6 &7 &8  \\
8 &2 &3 &4 &5 &6 &7 &1 
\end{array}
\right),\,
\tau=\left(
\begin{array}{cccccccc}
1 &2 &3 &4 &5 &6 &7 &8  \\
8 &7 &6 &5 &4 &3 &2 &1 
\end{array}
\right).$$
They act on $\XX_{C^5(8)^*}$ as
$\sigma(\overline{\lambda})=\overline{\lambda}$, $\sigma(\overline{\lambda'})=\overline{\lambda'}$ and $\tau(\overline{\lambda})=\overline{\lambda'}$.
Then we obtain the following proposition.

\begin{prop}\label{prop:small cover over C5(8)}
There exists only one small cover over $C^5(8)^*$ up to weakly equivariant homeomorphism.
\end{prop}

We also see that $p^{-1}(\overline{\lambda})/\langle\sigma\rangle$ corresponds with $\X_{C^5(8)^*}/\mathrm{Aut}\,(K_{C^5(8)^*})$.
Any element in $p^{-1}(\overline{\lambda})$ is represented by a characteristic matrix in the form
$$(\lambda_k;a,b):=
\arraycolsep5pt
\left(
\begin{array}{@{\,}c|ccccccc@{\,}}
1 &0 &0 &0 &0 &a &b &1\\
\hline
0&&&&&&&\\
0&\multicolumn{7}{c}{\raisebox{-10pt}[0pt][0pt]{\Huge $\lambda_k$}}\\
0&&&&&&&\\
0&&&&&&&\\
\end{array}
\right),$$
where $a$, $b$ are even integers and $\lambda_k\ (k=1,2,\ldots,28)$ is the characteristic matrix on $C^4(7)^*$ in the previous section.
There are sixty-four characteristic matrices in this form, and the action of $\sigma$ is illustrated as follows:

%
%
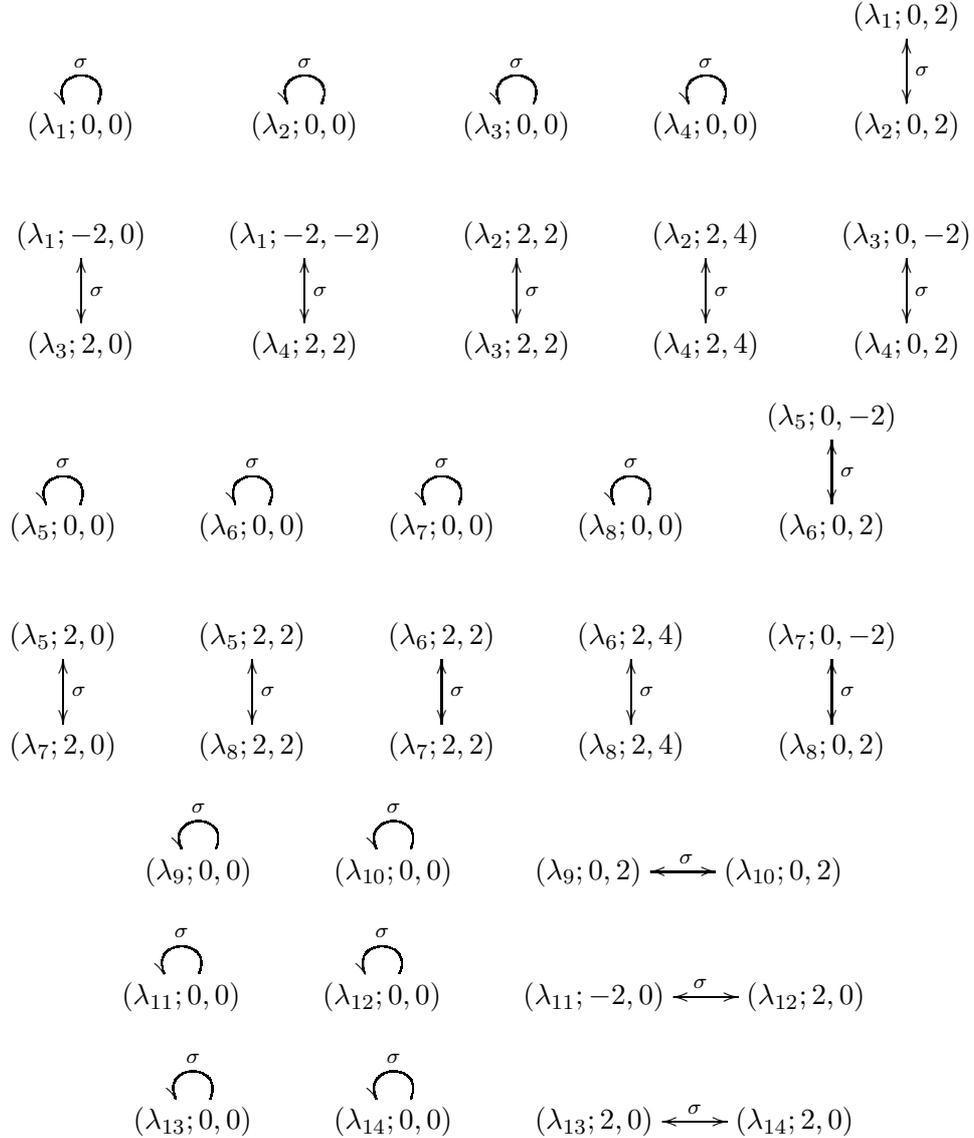
\[\xymatrix{&&&&(\lambda_1;0,2)\ar[d]^{\sigma}\\
(\lambda_1;0,0)\ar@(ur,ul)_{\sigma}&(\lambda_2;0,0)\ar@(ur,ul)_{\sigma}&(\lambda_3;0,0)\ar@(ur,ul)_{\sigma}&(\lambda_4;0,0)\ar@(ur,ul)_{\sigma}&
(\lambda_2;0,2)\ar[u]\\
(\lambda_1;-2,0)\ar[d]^{\sigma}&(\lambda_1;-2,-2)\ar[d]^{\sigma}&
(\lambda_2;2,2)\ar[d]^{\sigma}&(\lambda_2;2,4)\ar[d]^{\sigma}&(\lambda_3;0,-2)\ar[d]^{\sigma} \\
(\lambda_3;2,0)\ar[u]&(\lambda_4;2,2)\ar[u]&(\lambda_3;2,2)\ar[u]&(\lambda_4;2,4)\ar[u]&(\lambda_4;0,2)\ar[u]
}\]
\[\xymatrix{&&&&(\lambda_5;0,-2)\ar[d]^{\sigma}\\
(\lambda_5;0,0)\ar@(ur,ul)_{\sigma}&(\lambda_6;0,0)\ar@(ur,ul)_{\sigma}&(\lambda_7;0,0)\ar@(ur,ul)_{\sigma}&(\lambda_8;0,0)\ar@(ur,ul)_{\sigma}&
(\lambda_6;0,2)\ar[u]\\
(\lambda_5;2,0)\ar[d]^{\sigma}&(\lambda_5;2,2)\ar[d]^{\sigma}&
(\lambda_6;2,2)\ar[d]^{\sigma}&(\lambda_6;2,4)\ar[d]^{\sigma}&(\lambda_7;0,-2)\ar[d]^{\sigma} \\
(\lambda_7;2,0)\ar[u]&(\lambda_8;2,2)\ar[u]&(\lambda_7;2,2)\ar[u]&(\lambda_8;2,4)\ar[u]&(\lambda_8;0,2)\ar[u]&
}\]
\[\xymatrix{
(\lambda_9;0,0)\ar@(ur,ul)_{\sigma}&(\lambda_{10};0,0)\ar@(ur,ul)_{\sigma}&(\lambda_9;0,2)\ar[r]^{\sigma}&(\lambda_{10};0,2)\ar[l]
}\]
\[\xymatrix{
(\lambda_{11};0,0)\ar@(ur,ul)_{\sigma}&(\lambda_{12};0,0)\ar@(ur,ul)_{\sigma}&(\lambda_{11};-2,0)\ar[r]^{\sigma}&(\lambda_{12};2,0)\ar[l]
}\]
\[\xymatrix{
(\lambda_{13};0,0)\ar@(ur,ul)_{\sigma}&(\lambda_{14};0,0)\ar@(ur,ul)_{\sigma}&(\lambda_{13};2,0)\ar[r]^{\sigma}&(\lambda_{14};2,0)\ar[l]
}\]
\[\xymatrix{
(\lambda_{15};0,0)\ar@(ur,ul)_{\sigma}&(\lambda_{16};0,0)\ar@(ur,ul)_{\sigma}&(\lambda_{17};0,0)\ar@(ur,ul)_{\sigma}&(\lambda_{18};0,0)\ar@(ur,ul)_{\sigma}
}\]
\[\xymatrix{
(\lambda_{19};0,0)\ar@(ur,ul)_{\sigma}&(\lambda_{20};0,0)\ar@(ur,ul)_{\sigma}&(\lambda_{19};0,2)\ar[r]^{\sigma}&(\lambda_{20};0,2)\ar[l]
}\]
\[\xymatrix{
(\lambda_{21};0,0)\ar@(ur,ul)_{\sigma}&(\lambda_{22};0,0)\ar@(ur,ul)_{\sigma}&(\lambda_{21};0,2)\ar[r]^{\sigma}&(\lambda_{22};0,2)\ar[l]
}\]
\[\xymatrix{
(\lambda_{23};0,0)\ar@(ur,ul)_{\sigma}&(\lambda_{24};0,0)\ar@(ur,ul)_{\sigma}&(\lambda_{23};0,2)\ar[r]^{\sigma}&(\lambda_{24};0,2)\ar[l]
}\]
\[\xymatrix{
(\lambda_{25};0,0)\ar@(ur,ul)_{\sigma}&(\lambda_{26};0,0)\ar@(ur,ul)_{\sigma}&(\lambda_{27};0,0)\ar@(ur,ul)_{\sigma}&(\lambda_{28};0,0)\ar@(ur,ul)_{\sigma}
}\]

Hence we obtain the following proposition.

\begin{prop}\label{prop:qtmfd up to w.e.homeo over C5(8)}
There are exactly forty-six quasitoric manifolds over $C^5(8)^*$ up to weakly equivariant homeomorphism.
\end{prop}

We will show that the cohomology rings of these quasitoric manifolds are not isomorphic.

For a characteristic matrix $\xi=(\lambda_k;a,b)=(\lambda_{i,j})$ on $C^5(8)^*$,
let us denote $v'_{i_1}\cdots v'_{i_k}$ in Theorem \ref{thm:cohomology,particular} by $[i_1,\ldots,i_k]_\xi$ or simply by $[i_1,\ldots,i_k]$.
Put $X:=v_6,Y:=v_7,Z:=v_8$.
Then we obtain $H^*(\xi)\cong \Z[X,Y,Z]/(\I_\xi^6+\I_\xi^8)$,
where $\I_\xi^6$ and $\I_\xi^8$ are the ideals below:
\begin{align*}
\I_\xi^6 &=([2,4,6],[2,4,7],[2,5,7],[3,5,7]), \\
\I_\xi^8 &=([1,3,5,8],[1,3,6,8],[1,4,6,8]).
\end{align*}
Remark that $\I_\xi^6$ is determined by $\lambda_k$ (namely, independent of $a$ and $b$). Hence we write $\I_k^6=\I_\xi^6$.

Now we briefly explain the procedure for classifying the quasitoric manifolds over $C^5(8)^*$.
We will first show that there exist no more than a hundred automorphisms of $\Z[X,Y,Z]$ which map $\I^6_k$ into $\I^6_{k'}$ for some $k,k'=1,\ldots,28$.
Then we only have to verify that those automorphisms do not map $\I^8_{(\lambda_k;a,b)}$ into $\I^6_{k'}+\I^8_{(\lambda_{k'};a',b')}$
for any combination of even integers $a,b,a',b'$ by which $(\lambda_k;a,b)$ and $(\lambda_{k'};a',b')$ are characteristic matrices on $C^5(8)^*$.

For the convenience, we make the following definition.

\begin{df}
For $k,k'=1,2,\ldots,28$, we denote by $\A_{k,k'}$ the set of all graded ring automorphisms of $\Z[X,Y,Z]$ which maps $\I^6_k$ into $\I^6_{k'}$.
\end{df}

Let $\lambda=(\lambda_k;a,b)$ and $\lambda'=(\lambda_{k'};a',b')$ be characteristic matrices on $C^5(8)^*$ such that $\sigma(\lambda)=\lambda'$.
With the notations in Theorem \ref{thm:cohomology,particular},
the correspondence $v_i\mapsto v'_{\sigma(i)}\ (i=6,7,8)$ induces a graded ring automorphism $\phi_{(\sigma;\lambda,\lambda')}$ of $\Z[X,Y,Z]$
which maps $\I_\lambda$ into $\I_{\lambda'}$.
In particular, $\phi_{(\sigma;\lambda,\lambda')}$ maps $\I^6_k$ into $\I^6_{k'}$.
Then we have the following lemma.

\begin{lem}\label{lem:reducing of k,k'}
Let $\lambda_1=(\lambda_{k_1};a_1,b_1),\lambda'_1=(\lambda_{k'_1};a'_1,b'_1),
\lambda_2=(\lambda_{k_2};a_2,b_2),\lambda'_2=(\lambda_{k'_2};a'_2,b'_2)$ be characteristic matrices on $C^5(8)^*$.
If $\sigma(\lambda_i)=\lambda'_i$ for $i=1,2$, then the correspondence
$\A_{k_1,k_2}\rightarrow \A_{k'_1,k'_2}\colon \psi \mapsto \phi_{(\sigma;\lambda_2,\lambda'_2)}\circ\psi\circ\phi_{(\sigma;\lambda_1,\lambda'_1)}^{-1}$
is bijective.
\end{lem}

Let us define the equivalence relation $\sim_\sigma$ on $\{1,\ldots,28\}$ as follows:
$k\sim_\sigma k'$ if and only if $\sigma((\lambda_k;a,b))=(\lambda_{k'};a',b')$ for some $a,b,a',b'$
by which $(\lambda_k;a,b)$ and $(\lambda_{k'};a',b')$ are characteristic matrices on $C^5(8)^*$.
Then the equivalence classes by $\sim_\sigma$ are
$\{1,2,3,4\}$, $\{5,6,7,8\}$, $\{9,10\}$, $\{11,12\}$, $\{13,14\}$, $\{15\}$, $\{16\}$, $\{17\}$, $\{18\}$,
$\{19,20\}$, $\{21,22\}$, $\{23,24\}$, $\{25\}$, $\{26\}$, $\{27\}$, $\{28\}$.
Lemma \ref{lem:reducing of k,k'} implies that we only have to determine $\A_{k,k'}$ for
$$k,k'=1,5,9,11,13,15,16,17,18,19,21,23,25,26,27,28.$$

Hereafter we assume that a graded ring automorphism $\phi$ of $\Z[X,Y,Z]$ maps $\I^6_\lambda$ into $\I^6_{\lambda'}$,
where $\lambda$ and $\lambda'$ are characteristic matrices in the list. Write
$$\lambda=\left(
\begin{array}{ccccc|ccc}
1 &0 &0 &0 &0 &a &b &1  \\
0 &1 &0 &0 &0 &1 &p &q  \\
0 &0 &1 &0 &0 &x &y &1  \\
0 &0 &0 &1 &0 &1 &s &t  \\
0 &0 &0 &0 &1 &m &n &1 
\end{array}
\right),\,
\lambda'=\left(
\begin{array}{ccccc|ccc}
1 &0 &0 &0 &0 &a' &b' &1  \\
0 &1 &0 &0 &0 &1 &p' &q'  \\
0 &0 &1 &0 &0 &x' &y' &1  \\
0 &0 &0 &1 &0 &1 &s' &t'  \\
0 &0 &0 &0 &1 &m' &n' &1 
\end{array}
\right).$$
It is convenient to identify a homogeneous polynomial $f$ of degree $6$ in $\Z[X,Y,Z]$ with the row vector $u$ in such a way that
$$f=(X^3,Y^3,Z^3,X^2Y,XY^2,Y^2Z,YZ^2,Z^2X,ZX^2,XYZ){}^t u.$$
Then we see that ${}^t ([2,4,6]_\lambda,[2,4,7]_\lambda,[2,5,7]_\lambda,[3,5,7]_\lambda)$ is identified with the matrix below:
\begin{equation}\label{eq:generator}
\left(
\begin{array}{cccccccccc}
1 &0 &0 &p+s &ps &0 &0 &qt &q+t &pt+qs  \\
0 &ps &0 &1 &p+s &pt+qs &qt &0 &0 &t+q  \\
0 &pn &0 &m &pm+n &p+qn &q &0 &0 &1+qm  \\
0 &yn &0 &xm &ym+xn &y+n &1 &0 &0 &x+m 
\end{array}
\right),
\end{equation}
and this matrix is congruent with the following matrix modulo 2:
\begin{equation}\label{eq:gen,mod2}
\left(
\begin{array}{cccccccccc}
1 &0 &0 &1 &0 &0 &0 &0 &1 &1  \\
0 &0 &0 &1 &1 &1 &0 &0 &0 &1  \\
0 &0 &0 &0 &1 &1 &1 &0 &0 &1  \\
0 &1 &0 &0 &1 &0 &1 &0 &0 &1 
\end{array}
\right)
\end{equation}

As in Section 5, we denote $\phi$ by the matrix $\left(
\begin{array}{ccc}
a_1 &b_1 &c_1  \\
a_2 &b_2 &c_2  \\
a_3 &b_3 &c_3 
\end{array}
\right)$ such that $$\left(
\begin{array}{c}
\phi(X) \\
\phi(Y) \\
\phi(Z) 
\end{array}
\right)=\left(
\begin{array}{ccc}
a_1 &b_1 &c_1  \\
a_2 &b_2 &c_2  \\
a_3 &b_3 &c_3 
\end{array}
\right)\left(
\begin{array}{c}
X \\
Y \\
Z 
\end{array}
\right),$$
and put $a:=(a_1,a_2,a_3)$, $b:=(b_1,b_2,b_3)$, $c:=(c_1,c_2,c_3)$.

\begin{lem}\label{lem:phi mod2}
$\phi\equiv \left(
\begin{array}{ccc}
1 &0 &0  \\
0 &1 &0  \\
0 &0 &1 
\end{array}
\right) \bmod 2.$
\end{lem}
\begin{proof}
By row operations, we transform the matrix (\ref{eq:gen,mod2}) into
\begin{equation}\label{eq:gen,mod2,new}
\left(
\begin{array}{cccccccccc}
1 &0 &0 &0 &0 &0 &1 &0 &1 &1  \\
0 &0 &0 &1 &0 &0 &1 &0 &0 &0  \\
0 &0 &0 &0 &1 &1 &1 &0 &0 &1  \\
0 &1 &0 &0 &0 &1 &0 &0 &0 &0 
\end{array}
\right).
\end{equation}
Since the coefficients of $Z^3$ in $\phi([2,4,6])$, $\phi([2,4,7])$, $\phi([2,5,7])$, $\phi([3,5,7])$ are equal to zero (see the matrix (\ref{eq:generator})),
we obtain
\begin{equation}\label{eq:Z^3}
\left(
\begin{array}{c}
(c_1+pc_2+qc_3)(c_1+sc_2+tc_3)c_1 \\
(c_1+pc_2+qc_3)(c_1+sc_2+tc_3)c_2 \\
(c_1+pc_2+qc_3)(mc_1+nc_2+c_3)c_2 \\
(xc_1+yc_2+c_3)(mc_1+nc_2+c_3)c_2 
\end{array}
\right) \\=0.
\end{equation}
\begin{itemize}
 \item If $c_1\equiv 1,c_2\equiv 0 \bmod 2$, we have $c_3\equiv 1 \bmod 2$.
 Since the coefficient of $Z^2X$ in $\phi([2,4,6])$ is congruent to zero modulo $2$, we obtain $a_1+a_3\equiv 0 \bmod 2$.
 Due to the linear independence of $a,b,c$, we have $a_2\equiv 1,b_1+b_3\equiv 1 \bmod 2$.
 Besides, since the coefficients of $X^3$ and $ZX^2$ in $\phi([3,5,7])$ are congruent modulo 2,
 we have $a_1\equiv 0 \bmod 2$, so $a\equiv (0,1,0) \bmod 2$.
 Then
 \begin{align*}
 \phi([2,4,6]) &\equiv Y(X+(b_1+b_2)Y)(b_1Y+Z) \\
  &\equiv (0,b_1(b_1+b_2),0,0,b_1,b_2,1,0,0,1) \bmod 2
 \end{align*}
 has zero in the coefficient of $X^3$, so the coefficients of $XY^2$ and $XYZ$ are congruent modulo 2.
 Hence we have $b_1\equiv 1 \bmod 2$, and then $b_3\equiv 0 \bmod 2$.
 However,
 \begin{align*}
 \phi([2,5,7]) &\equiv Y(X+b_2Y+Z)(X+b_2Z) \\
  &\equiv (0,b_2,0,1,0,b_2,0,0,0,1) \bmod 2
 \end{align*}
 is not a linear combination of $[2,4,6],[2,4,7],[2,5,7],[3,5,7]$.
 \item If $c_1\equiv 0,c_2\equiv 1 \bmod 2$,
 the second and fourth components of the left hand side of (\ref{eq:Z^3}) can not be congruent to zero modulo 2 simultaneously.
 \item If $c_1\equiv c_2\equiv 1 \bmod 2$, due to the equation (\ref{eq:Z^3}), we have $c_3\equiv 1 \bmod 2$.
 From the coefficient of $Z^2X$ in $\phi([3,5,7])$, we obtain $a_2\equiv a_3 \bmod 2$.
 Then $a_1\not\equiv a_2 \bmod 2$ due to the linear independence of $a$ and $c$.
 Similarly, from the coefficients of $X^3$ and $ZX^2$ in $\phi([2,4,6])$ and $\phi([2,4,7])$, we obtain $a_1\equiv a_2\equiv 1 \bmod 2$.
 This is a contradiction.
 \item If $c_1\equiv c_2\equiv 0 \bmod 2$, then we have $c_3\equiv 1 \bmod 2$.
 From the coefficients of $Z^2X$ in $\phi([2,5,7])$, we obtain $a_2\equiv 0 \bmod 2$.
 Besides, from $\det \phi\equiv 1 \bmod 2$, we obtain $a_1\equiv b_2 \equiv 1 \bmod 2$.
 Comparing the coefficients of $X^3$ and $ZX^2$ in $\phi([2,4,6])$, we obtain $a_3\equiv 0 \bmod 2$.
 Then
 \begin{align*}
 \phi([2,4,7]) &\equiv (X+(b_1+b_3)Y+Z)(X+(b_1+1)Y)Y \\
  &\equiv (0,(b_1+b_3)(b_1+1),0,1,b_3+1,b_1+1,0,0,0,1) \bmod 2
 \end{align*}
 has zero in the coefficient of $X^3$, so the coefficient of $XY^2$ and $XYZ$ in $\phi([2,4,7])$ are congruent.
 Hence we have $b_3\equiv 0 \bmod 2$, and then $\phi([2,4,7])$ is needed to be congruent to $(0,0,0,1,1,1,0,0,0,1)$ modulo 2,
 so we obtain $b_1\equiv 0 \bmod 2$.
\end{itemize}
Thus the proof is completed.
\end{proof}

By Lemma \ref{lem:phi mod2}, $(xc_1+yc_2+c_3)(mc_1+nc_2+c_3) \equiv 1 \bmod 2$.
Moreover, since $q \equiv 1 \bmod 2$, we have $c_1+pc_2+qc_3\equiv 1 \bmod 2$.
Hence we obtain the following corollary from the equation (\ref{eq:Z^3}).

\begin{cor}\label{cor:c2=0}
$c_2=(c_1+tc_3)c_1=0$.
\end{cor}

Comparing the coefficients of $X^3$ and $Z^2X$ in $\phi([2,4,6])$, $\phi([2,4,7])$, $\phi([2,5,7])$, $\phi([3,5,7])$, we obtain
\begin{equation}\label{eq:X^3 Z^2X}
\begin{aligned}
q't'&\left(
\begin{array}{c}
(a_1+pa_2+qa_3)(a_1+sa_2+ta_3)a_1 \\
(a_1+pa_2+qa_3)(a_1+sa_2+ta_3)a_2 \\
(a_1+pa_2+qa_3)(ma_1+na_2+a_3)a_2 \\
(xa_1+ya_2+a_3)(ma_1+na_2+a_3)a_2 
\end{array}
\right) \\
&=\left(
\begin{array}{c}
(c_1+qc_3)\{(c_1+tc_3)a_1+(a_1+sa_2+ta_3)c_1\} \\
(c_1+qc_3)(c_1+tc_3)a_2 \\
(c_1+qc_3)(mc_1+c_3)a_2 \\
(xc_1+c_3)(mc_1+c_3)a_2 
\end{array}
\right).
\end{aligned}
\end{equation}
If we assume $a_2\neq 0$, from the equation (\ref{eq:X^3 Z^2X}), we obtain $q't'(xa_1+ya_2+a_3)(ma_1+na_2+a_3)=(xc_1+c_3)(mc_1+c_3)$.
The left hand side of this equation is even but the right hand side is odd.
Hence we have the following lemma.

\begin{lem}\label{lem:a2=0}
$a_2=0.$
\end{lem}

By Corollary \ref{cor:c2=0} and Lemma \ref{lem:a2=0}, we obtain the following proposition.

\begin{prop}
$(a_2,b_2,c_2)=\pm (0,1,0)$.
\end{prop}

Considering $-\phi$ if necessary, we can assume $b_2=1$.

First, we consider the case that $c_1\neq 0$.
From the equation (\ref{eq:Z^3}), we have $c_1+tc_3=0$.
Hence $t\neq 0$ (remark that, due to the list of characteristic matrices, this means $t=2$)
and $c=c_3(-2,0,1)$, where $c_3=\pm 1$.
From the equation (\ref{eq:X^3 Z^2X}), we have $$q't'(a_1+qa_3)(a_1+2a_3)a_1=-2(q-2)(a_1+2a_3).$$
Since $a_1+2a_3\neq 0$ by Lemma \ref{lem:phi mod2}, this equation reduces to $$q't'(a_1+qa_3)a_1=-2(q-2).$$
Moreover, since the right hand of this equation is not zero, we have $t'=t=2$ and
\begin{equation}\label{eq:C'}
q'(a_1+qa_3)a_1=2-q.
\end{equation}
Comparing the coefficients of $X^3$ and $ZX^2$ in $\phi([2,4,6])$, we obtain
$$(q'+2)(a_1+qa_3)(a_1+2a_3)a_1=(a_1+2a_3)\{ (a_1+qa_3)(-2c_3)+(q-2)c_3a_1 \}.$$
Since $a_1+2a_3\neq 0$, this equation can be shortened to
\begin{equation}\label{eq:D'}
(q'+2)(a_1+qa_3)a_1=\{-2(a_1+qa_3)+(q-2)a_1\}c_3.
\end{equation}
Since $\det \phi=\pm 1$, we obtain $c_3(a_1+2a_3)=\pm 1$.
Putting $\epsilon=c_3 \det \phi$, we have $a_1+2a_3=\epsilon$.
Then the equation (\ref{eq:C'}) is rewritten as
$$(\epsilon+(q-2)a_3)(\epsilon-2a_3)q'=2-q.$$
\begin{itemize}
 \item If $q=1$ or $3$, $\epsilon-2a_3$ divides $1$. Since $a_3$ is even, we obtain $a_3=0$.
 \item If $q=-1$, $\epsilon-3a_3$ divides $3$. Hence we obtain $a_3=0$.
\end{itemize}
In any case, we obtain $a_3=0$, implying $q'=2-q$.
From the equation (\ref{eq:D'}), we have $q'+2=-(q'+2)a_1c_3$, namely, $a_1c_3=-1$.
Thus we obtain the following lemma.

\begin{lem}
If $c_1\neq 0$, we have $t'=t=2$, $q'=2-q$ and
$$\phi=\pm\left(
\begin{array}{ccc}
a_1 &b_1 &2a_1  \\
0 &1 &0  \\
0 &b_3 &-a_1 
\end{array}
\right).$$
\end{lem}

Then we consider the cases (1) $(q,q')=(3,-1)$, (2) $(q,q')=(-1,3)$ and (3) $(q,q')=(1,1)$ separately.

\begin{enumerate}
 \item[(1)] If $q=3$ and $q'=-1$, then $\lambda=(\lambda_{28};0,0)$ and $\lambda'=(\lambda_{26};0,0),(\lambda_{27};0,0)$. 
 The matrix (\ref{eq:generator}) for $\lambda'$ reduces to
 $$\left(
\begin{array}{cccccccccc}
1 &0 &0 &1 &0 &0 &0 &-2 &1 &-1  \\
0 &0 &0 &1 &1 &-1 &-2 &0 &0 &1  \\
0 &0 &0 &m' &1 &-1 &-1 &0 &0 &1-m'  \\
0 &1 &0 &m' &m'+1 &2 &1 &0 &0 &1+m' 
\end{array}
\right).$$
 Since $\phi([2,4,6]-b_1[2,4,7])$ has zero in the coefficient of $Y^3$,
 the coefficients of $XY^2$ and $Y^2Z$ in $\phi([2,4,6]-b_1[2,4,7])$ are negatives of each other.
 Then we have $$a_1(b_1+3b_3+2)(b_1+2b_3+1)=-2a_1(b_1+3b_3+2)(b_1+2b_3+1).$$
 Since $a_1,b_1+2b_3+1\equiv 1\bmod 2$, this equation reduces to $b_1+3b_3+2=0$.
 Then
 \begin{align*}
 &\phi([2,4,6]-b_1[2,4,7]) \\
 &=(X-Z)(a_1X-(b_3+1)Y)(X+2Z) \\
 &=(a_1,0,0,-b_3-1,0,0,2(b_3+1),-2a_1,a_1,-b_3-1) \\
 &\equiv (0,0,0,-a_1-b_3-1,0,0,2(b_3+1),0,0,a_1-b_3-1) \bmod \I^6_{\lambda'}.
 \end{align*}
 Hence there exists an integer $\alpha$ such that $$\alpha (1-m')=-a_1-b_3-1,-\alpha=2(b_3+1),\alpha m'=a_1-b_3-1.$$
 We obtain $(b_3+1)(1-2m')=a_1=\pm 1$ from these equations. Then we have $m'=0,\ b_3=a_1-1$.
 Therefore we see that $$\phi=\left(
\begin{array}{ccc}
1 &-2 &2  \\
0 &1 &0  \\
0 &0 &-1 
\end{array}
\right)\,\mathrm{or}\,\left(
\begin{array}{ccc}
-1 &4 &-2  \\
0 &1 &0  \\
0 &-2 &1 
\end{array}
\right).$$
 In each case,
 \begin{align*}
 \phi([2,5,7]) &=(X-Z)(Y-Z)Y \\
 &=(0,0,0,0,1,-1,1,0,0,-1)\notin \I^6_{\lambda'}.
\end{align*}
 Hence we see that there exists no graded ring isomorphism from $H^*(\lambda)$ to $H^*(\lambda')$ in this case.
 \item[(2)] If $q=-1$ and $q'=3$, from the conclusion of the previous case, we see that there exists no graded ring isomorphism from $H^*(\lambda)$ to $H^*(\lambda')$ 
 \item[(3)] If $q=q'=1$, then we can assume $\lambda,\lambda'=(\lambda_{19};0,0)$, $(\lambda_{21};0,0)$, $(\lambda_{23};0,0)$, $(\lambda_{25};0,0)$.
 In the way similar to the case (1),
 we can show that all graded ring automorphisms such that they map $\I^6_\lambda$ into $\I^6_{\lambda'}$ and $c_1\neq 0$ are in the illustration below.
 \[\xymatrix {
 (\lambda_{19};0,0)\ar@(ur,ul)_{\pm\phi_1} & (\lambda_{21};0,0)\ar@(ur,ul)_{\pm\phi_1,\pm\phi_2}\ar[rr]^{\pm\phi_1,\pm\phi_2}&&
  (\lambda_{23};0,0)\ar@(ur,ul)_{\pm\phi_1,\pm\phi_2}\ar[ll]
 }\]
 Here $\phi_1=\left(
\begin{array}{ccc}
1 &2 &2  \\
0 &1 &0  \\
0 &-2 &-1 
\end{array}
\right)$ and $\phi_2=\left(
\begin{array}{ccc}
-1 &0 &-2  \\
0 &1 &0  \\
0 &0 &1 
\end{array}
\right)$.
\end{enumerate}

Next, in the case that $c_1=0$, by the same method as above,
we can show that all non-trivial graded ring automorphisms such that they map $\I^6_\lambda$ into $\I^6_{\lambda'}$ and $c_1=0$ are in the illustration below.
\[\xymatrix {
(\lambda_1;0,0)\ar@(ur,ul)_{\pm\phi_4}\ar[d]^{\pm \phi_5,\pm \phi_6} & (\lambda_{11};0,0)\ar[d]^{\pm\phi_5} &
 (\lambda_{23};0,0)\ar@(ur,ul)_{\pm\phi_1}\ar[d]^{\pm \mathrm{id},\pm\phi_7} \\
(\lambda_5;0,0)\ar@(dl,dr)_{\pm\phi_3}\ar[u] & (\lambda_{13};0,0)\ar[u] &
 (\lambda_{25};0,0)\ar@(dl,dr)_{\pm\phi_7}\ar[u]
}\]
Here $\phi_3,\ldots,\phi_7$ denote the matrices below:
$$\phi_3=\left(
\begin{array}{ccc}
-1 &-2 &0  \\
0 &1 &0  \\
0 &2 &-1 
\end{array}
\right),\phi_4=\left(
\begin{array}{ccc}
-1 &-2 &0  \\
0 &1 &0  \\
0 &-2 &-1 
\end{array}
\right),\phi_5=\left(
\begin{array}{ccc}
1 &0 &0  \\
0 &1 &0  \\
0 &0 &-1 
\end{array}
\right),
$$
$$\phi_6=\left(
\begin{array}{ccc}
-1 &-2 &0  \\
0 &1 &0  \\
0 &2 &1 
\end{array}
\right),\phi_7=\left(
\begin{array}{ccc}
-1 &2 &0  \\
0 &1 &0  \\
0 &-2 &-1 
\end{array}
\right).
$$

Now we have all graded ring automorphisms of $\Z[X,Y,Z]$ which map $\I_{\xi}^6$ into $\I_{\xi'}^6$
with respect to some characteristic matrices $\xi$ and $\xi'$ on $C^5(8)^*$, where $\xi$ and $\xi'$ are not equivalent.
In fact, none of these automorphisms maps $\I_{\xi}^8$ into $\I_{\xi'}$.
One can verify it by a computer-assisted calculation.
Thus we have the topological classification of quasitoric manifolds over $C^5(8)^*$.

\begin{thm}\label{thm:qtmfd over C5(8) up to homeo}
There are exactly forty-six quasitoric manifolds over $C^5(8)^*$ up to homeomorphism, and they are distinguished by their cohomology rings.
\end{thm}

%
%

\section*{ACKNOWLEDGEMENT}

The author is deeply grateful to Daisuke Kishimoto for his support and valuable advice.

\end{document}